\documentclass[hidelinks,onefignum,onetabnum]{siamart250211}

\usepackage{mathtools}
\usepackage{multirow}
\usepackage{subcaption}
\usepackage{svg}
\usepackage{multirow}
\usepackage{comment}

\usepackage[dvipsnames]{xcolor}

\usepackage{lipsum}
\usepackage{amsfonts}
\usepackage{graphicx}
\usepackage{epstopdf}
\usepackage{algorithmic}
\ifpdf
  \DeclareGraphicsExtensions{.eps,.pdf,.png,.jpg}
\else
  \DeclareGraphicsExtensions{.eps}
\fi

\newsiamremark{remark}{Remark}
\newsiamremark{hypothesis}{Hypothesis}
\crefname{hypothesis}{Hypothesis}{Hypotheses}
\newsiamthm{claim}{Claim}
\newsiamremark{fact}{Fact}
\crefname{fact}{Fact}{Facts}

\headers{%
  \parbox[t]{1\textwidth}{%
    Learning dynamically inspired bases for Koopman \\and transfer operator approximation}%
}{G. Froyland and K. Kühl}

\title{An Example Article\thanks{Submitted to the editors DATE.
\funding{This work was funded by the Fog Research Institute under contract no.~FRI-454.}}}

\usepackage{amsopn}

\ifpdf
\hypersetup{
  pdftitle={Learning dynamically inspired bases for Koopman and transfer operator approximation},
  pdfauthor={G. Froyland, K. Kühl}
}
\fi

\title{Learning dynamically inspired bases for Koopman and transfer operator approximation}

\author{Gary Froyland\thanks{School of Mathematics and Statistics, UNSW Sydney, Sydney NSW 2052, Australia.
(\email{g.froyland@unsw.edu.au}, \email{k.kuhl\_oliveira@unsw.edu.au}).}
\and Kevin Kühl\footnotemark[1]}

\begin{document}

\maketitle

\begin{abstract}
Transfer and Koopman operator methods offer a framework for representing complex, nonlinear dynamical systems via linear transformations, enabling a deeper understanding of the underlying dynamics. 
The spectra of these operators provide important insights into system predictability and emergent behaviour, although efficiently estimating them from data can be challenging.
We approach this issue through the lens of general operator and representational learning, in which we approximate these linear operators using efficient finite-dimensional representations. Specifically, we machine-learn orthonormal basis functions that are dynamically tailored to the system. This learned basis provides a particularly accurate approximation of the operator's action and enables efficient recovery of eigenfunctions and invariant measures.
We illustrate our approach with examples that showcase the retrieval of spectral properties from the estimated operator, and emphasise the dynamically adaptive quality of the machine-learned basis.
\end{abstract}



\section{Introduction}
Nonlinear dynamical systems emerge in a variety of domains, including biology, ecology, physics, and engineering.
Understanding their behaviour is critical for prediction and control, but analysing nonlinear systems is difficult due to their complex and frequently high-dimensional chaotic behaviour. A promising approach is to study a \textit{nonlinear} system via an associated \textit{linear} operator acting on observables (\emph{nonlinear functions} of the state). In this operator-theoretic framework, the nonlinear evolution in state space is lifted to a linear (but infinite-dimensional) evolution in a function space. For example, given a discrete dynamical system $x_{n+1} = T(x_n)$ with state space $X$, one can define the transfer operator $\mathcal{L}$ acting on a function $g: X \to \mathbb{R}$ by composition:
\begin{equation*}
    \mathcal{L}g \coloneq g\circ T^{-1}.
\end{equation*}

The operator $\mathcal{L}$ is linear even if $T$ is nonlinear. By analysing $\mathcal{L}$, one can leverage linear techniques to understand nonlinear dynamics \cite{Froyland2021, Mezi2013, pmlr-v44-giannakis15, Eisner2015, Klus2018}. The trade-off is that $\mathcal{L}$ typically acts on infinite-dimensional function spaces, and to study $\mathcal{L}$ in practice we must work with finite-dimensional approximations.

A major challenge in typical applications is that we do not know \textit{a priori} which observables or subspaces will yield a good finite-dimensional representation of $\mathcal{L}$. Many traditional approaches assume a predefined set of basis functions --- e.g.\ monomials, Fourier modes, or the library of observables supplied to Extended Dynamic Mode Decomposition (EDMD) \cite{Williams2015} --- that may not be well suited to the system at hand. For example, data-driven approximations of the Koopman operator, such as DMD \cite{SCHMID_2010, dmd_mezic} and its variations, depend on the user supplying observables that span a subspace that is adapted to the dynamics of $\mathcal{L}$. If key observables are omitted, important dynamics can be missed, and poorly chosen subspaces can lead to significant estimation errors. Consequently, identifying coordinate transformations or feature spaces in which a nonlinear system behaves linearly remains an active and widely researched topic \cite{koopman31, amini_carleman_linearization, nonlinear_pde_linear_pde_coordinate, Lee2022}.

Several approaches attempt to learn such coordinate transformations. Brunton \textit{et al.}\ consider a sparse dictionary of functions with fixed form to obtain a finite Koopman-invariant subspace for control \cite{Brunton2016}. Starting from a specially structured dictionary, Johnson \textit{et al.}\ \cite{Yeung2022} are able to construct invariant subspaces and, in a companion work, propose a test for their invariance \cite{Johnson2025}.
The ResDMD framework of Colbrook and Townsend \cite{CT23} uses residuals to assess eigenspace invariance. Turning to the use of neural networks, Li \emph{et al.}\ \cite{BolltQianxiao2017} employ a neural dictionary (neural networks as functions), updating the basis and Koopman matrix separately. As a result, operator errors never backpropagate into the dictionary, and the learned observables are not adapted to the underlying dynamics.

Unifying both operator and embedding learning, Otto and Rowley's linearly recurrent autoencoder network (LRAN) forces an encoded version of the state space to evolve linearly in a Koopman-invariant subspace \cite{rowleyOtto2019}. However, the loss is skewed towards a good reconstruction of the state space from its encoded version. The network is tuned to embed the state space compactly rather than to learn a function space that is adapted to the underlying dynamics for accurate spectral reconstruction.
Other encoder-decoder schemes, such as Yeung \textit{et al.}\ \cite{Yeung2019} and control-oriented work by \cite{ShiMeng2022} follow a similar approach. Takeishi \textit{et al.}\ \cite{takeishi_invariant_subspaces} also learn observables by minimizing a one-step residual; at the global optimum these observables span an invariant subspace. Using a modified autoencoder structure, Lusch, Brunton, and Kutz consider the spectrum of the operator and learn Koopman eigenfunctions \cite{lusch_deep_learning_embeddings}, but tuning the network parameters is not straightforward.

In parallel, there has been significant progress in learning nonlinear operators (nonlinear maps between function spaces) with neural networks. This work is built upon the universal approximation theorem for operators \cite{universal_approximation_chen_chen, kovachkiNeuralOperatorLearning2024}, which extends the classic universal approximation property of neural nets from finite-dimensional functions to infinite-dimensional functionals.
Lu \textit{et al.}\ \cite{Lu2021} have proposed the Deep Operator Network (DeepONet) architecture. DeepONet consists of two subnetworks (a branch network that encodes the input function via evaluations at a fixed set of sensor points, and a trunk network that encodes the location where the output is evaluated) 
and can learn nonlinear operator mappings. Another example is the Basis Operator Network (BasisONet) of Hua and Lu \cite{HUA202321}, which uses a neural-network-based dual autoencoder setting to learn a lower-dimensional representation (basis) for functions in the input and output function spaces. Over the past few years, different architectures have been proposed for various scenarios, including Fourier Neural Operators \cite{fourier_neural_operators}, and PCA-Net \cite{pca_net}.
A comprehensive review of operator learning techniques may be found in \cite{mathematical_guide_operator_learning}, including insightful parallels to the matrix recovery problem.

In this current work we take the BasisONet architecture as a starting point, and exploit properties of Perron--Frobenius and Koopman operators to reduce the number of encoding blocks and produce a \textit{single basis of functions that span a dynamically adapted subspace}.
Our Single Autoencoder Basis Operator Network (SABON) learns a basis with no predefined functional form and trains it with a relative $L^2$ loss. For every training observable we minimise a residual measuring how well a finite matrix represents $\mathcal L$ on the subspace spanned by the learned basis.
Because of the inner-product structure of our projection block, gradient descent methods naturally push the basis toward orthogonality.
Orthogonality also improves the conditioning of the loss landscape, making training more stable. An optional sparsity penalty encourages local support (Example 1). The combined learning of both the operator estimate and the basis functions produces a basis that is adapted to the underlying dynamics (Examples 2 and 3).

An outline of the paper is as follows. \Cref{sec:methodology} introduces the Single Autoencoder Basis Operator Network (SABON) and proves a universal-approximation theorem for the presented setting.
\Cref{sec:examples} first validates SABON on a circle rotation map, and shows that spectral properties of the transfer operator can be recovered from the learned basis and approximator.
Next, we apply SABON to two nonlinear chaotic maps of the 2-torus and demonstrate that the learned basis adapts to the anisotropic dynamics in both cases, and outperforms equal-cardinality Fourier bases in approximating spectral quantities such as the Sinai--Ruelle--Bowen (SRB) measure.
The final section discusses our main findings and sketches directions for future work.

\section{Methodology}
\label{sec:methodology}
Let $X\subset\mathbb R^d$ denote a compact, metrizable space of intrinsic dimension $d' \le d$, where elements of $X$ are parameterised by $d$-tuples of real numbers. The set $X$ is the domain of our nonlinear dynamics, which are generated by the map $T\colon X\to X$. We assume that $\mathcal L$ is continuous; sufficient conditions on $T$ ensuring this continuity are given in \cref{remark_theorem}.
 Let $\{x_i\}_{i=1, \dots, n}\subset X$ be a set of points representing a discretization of the state space. 
 We assume that in the large-data limit these points are distributed according to a probability measure $\nu$ on $X$, and in what follows we consider the space of square-integrable real-valued functions $L^2(X,\nu)$, which we write as $L^2(X)$ for brevity.
 In our numerical experiments we will take $\nu$ to be normalised Lebesgue measure.
 An observable $g\in L^2(X)$ is represented by its values on the discrete points $\{x_i\}_{i=1, \dots, n}$; that is $[g(x_1), g(x_2),\ldots,g(x_n)]^\top$. The action of the transfer operator produces a new function $\mathcal{L}(g)$, whose values on the grid are $\big[(\mathcal{L}g)(x_1),\ldots,(\mathcal{L}g)(x_n)\big]^\top$. If $V \subset L^2(X)$ is a finite-dimensional invariant subspace spanned by basis functions ${\phi_1,\dots,\phi_N}$, then the restriction of $\mathcal{L}$ to $V$ can be represented by an $N\times N$ matrix $L$ such that $\mathcal{L}\phi_j = \sum_{k=1}^N L_{kj}\phi_k$. 
 If $V$ is not $\mathcal{L}$-invariant, a finite-rank approximation $\hat{\mathcal{L}}$ is often obtained by first applying $\mathcal{L}$ to the basis elements of $V$ and then projecting the resulting functions back into $V$. 
 The quality of this approximation depends on the choice of subspace $V$.
 Two important contributors to approximation accuracy are the richness of the space $V$ and the error incurred due to projection into $V$.
 In our learning procedure, a rich $V$ will be provided by a rich training set and by the orthogonality property of our learned basis, while our architecture and loss function explicitly incorporates the projection error.

 In preparation for the neural learning, we suppose that the state space $X$ may be smoothly embedded in $\mathbb{R}^d$.
 We will adopt the standard convention of scalar feed-forward neural networks as a composition of affine maps and nonlinear activations,
$\phi\colon\mathbb{R}^{d}\to\mathbb{R}$. Stacking $N$ such networks gives a vector‑valued network $\hat{\phi}(x)\colon\mathbb{R}^{d}\to\mathbb{R}^{N}$, written as 
 \begin{equation}
 \label{eq:vector_valued_neural}
 \hat{\phi}(x)=\bigl(\phi_{1}(x),\dots,\phi_{N}(x)\bigr)^{\top}
 \end{equation}
Finally, we call the collection $\{\phi_{1},\dots,\phi_{N}\}$ a \textit{neural basis} for the finite‑dimensional subspace $V=\operatorname{span}\{\phi_{j}\}_{j=1}^{N}\subset L^{2}(X)$.
 
\subsection{Single Autoencoder Basis Operator Networks (SABON)}
\label{subsec:sabon}

To address the problem of simultaneously learning an approximator to the transfer operator and a subspace adapted to the underlying dynamics, we employ an architecture comprising four stages: (i) an encoder $\mathcal E$ that constructs basis functions evaluated on a training grid, (ii) a block $\mathcal P$ that projects observables onto the learned basis, (iii) a linear latent map $\mathcal G$ that approximates the transfer operator in the coefficient space, and (iv) a reconstructor $\mathcal R$ that lifts these coefficients back to the function space by expanding them in the learned basis.

\begin{figure}[H]
\centering
\includegraphics[width=1\linewidth]{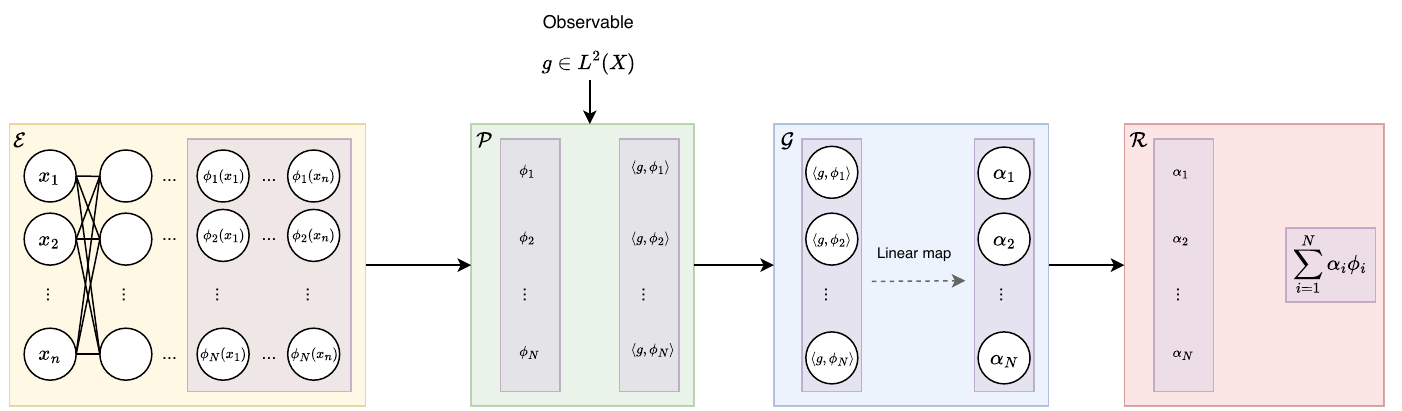}
 \caption{SABON consists of an encoder $\mathcal{E}$, a latent projection $\mathcal{P}$, a linear map $\mathcal{G}$, and a reconstruction map $\mathcal{R}$.
The encoder learns basis functions $\{\phi_j\}_{j=1}^{N}$, while $\mathcal{G}\colon\mathbb{R}^N\to\mathbb{R}^N$ approximates the action of $\mathcal{L}$ on $V=\operatorname{span}\{\phi_j\}_{j=1}^{N}$.
The reconstruction map recovers a function from its coefficients in the latent space.}
\label{fig:architecture}
\end{figure}

Two inputs are needed for the above architecture: the discretization $\{x_i\}_{i=1, \dots, n}\subset X$ where all maps are evaluated, and the observable function $g$, represented by its evaluation over the points $[g(x_1), g(x_2),\ldots,g(x_n)]^\top$.
We explain the SABON components in turn.

\paragraph{Encoder ($\mathcal{E}$)}
    The encoder in SABON is a stack of vector‑valued neural networks, each applied pointwise to one grid element. Recall from \eqref{eq:vector_valued_neural} that $\hat{\phi}\colon \mathbb{R}^{d}\longrightarrow \mathbb{R}^{N}$ is the vector-valued network defined in the embedding of $X$ in $\mathbb{R}^d$. The encoder
    \begin{equation*}
    \mathcal{E}\colon (\mathbb{R}^{d})^{n}\longrightarrow (\mathbb{R}^{N})^{n}\cong\mathbb{R}^{Nn},\qquad
  \mathcal{E}\bigl((x_1,\dots,x_n)\bigr)=\bigl(\hat{\phi}(x_1),\dots,\hat{\phi}(x_n)\bigr)
    \end{equation*}
acts by evaluating $\hat{\phi}$ at every grid point. Written componentwise, the output is the column-major vector
    \begin{equation*}
\bigl[\phi_1(x_1),\dots,\phi_N(x_1),\phi_1(x_2),\dots,\phi_N(x_2),\dots,\phi_1(x_n), \dots, \phi_N(x_n)\bigr]^{\top}\in\mathbb{R}^{Nn}.
    \end{equation*} 
 The collection $\{\phi_j\}_{j=1}^{N}$ forms the learned neural basis of the subspace $V$ introduced above.

\paragraph{Projection ($\mathcal{P}$)}
 The projection block computes\footnote{in practice, this inner product is evaluated as $\frac{1}{n}\sum_{i=1}^n g(x_i)\phi_j(x_i)$.} $\{\langle g, \phi_j \rangle \}_{j=1, \dots, N}$ for an observable $g \in L^2(X)$ and a set of learned basis functions $\{\phi_j\}_{j=1}^{N}$.
 We define $\mathcal{P}_{\{\phi_j\}_{j=1}^{N}}: L^2(X)\to \mathbb{R}^N$ via
\begin{equation}
\label{eq:projector_block}
    \mathcal{P}_{\{\phi_j\}_{j=1}^{N}}(g) =\bigl(\langle g,\phi_1\rangle,\dots,\langle g,\phi_N\rangle\bigr)
\end{equation}
We shall sometimes omit the subscript of $\mathcal{P}$.

\paragraph{Map $\mathcal{G}$ in latent space}
The action of the transfer/Koopman operator is represented by a linear map implemented in block $\mathcal{G}$. The coefficients $\langle g, \phi_j \rangle$, $j=1,\ldots,N$ in the latent space are fed into $\mathcal{G}$, which learns the action of $\hat{\mathcal{L}}$ in the latent space. That is, the block $\mathcal{G}$ consists of a learnable linear map capable of transforming the coefficients of $g$ when represented in the basis $\{\phi_j\}_{j=1}^{N}$ into the corresponding coefficients of $\hat{\mathcal{L}}(g)$ in the same basis.

\paragraph{Reconstruction ($\mathcal{R}$)}

The architecture employs a reconstruction block that converts the latent coefficients produced by the linear map $\mathcal{G}$ into the approximation of the transformed observable
\begin{equation*}
  \hat{\mathcal{L}}(g)(x)=\sum_{j=1}^{N}\alpha_j\,\phi_j(x).
\end{equation*}
Because the same functional expression can reconstruct any element of the span of the learned basis, we use the general definition of $\mathcal{R}_{\{\phi_j\}_{j=1}^{N}}\colon \mathbb{R}^{N}\to L^{2}(X)$ as
\begin{equation}
\label{eq:reconstructor_block}
  \mathcal{R}_{\{\phi_j\}_{j=1}^{N}}(a_1,\dots,a_N)=\sum_{j=1}^{N}a_j\,\phi_j.
\end{equation}
To simplify the notation, we will also sometimes omit the subscript in the sequel. The coefficients serving as input to $\mathcal{R}$ will always be clear from the context (i.e.\ which maps are composed with $\mathcal{R}$).

One of the key features of SABON is that it simultaneously optimises the learned subspace and the operator action. This is performed through a suitable loss function comprised of two components (in practice these are calculated with their discrete counterparts using evaluations at the points $\{x_i\}_{i=1}^n$):
\begin{enumerate}
    \item The composition $\mathcal{R}\circ\mathcal{G}\circ\mathcal{P}$, which approximates the projected action of $\mathcal{L}$ on $V$, has the loss term:
    
    \begin{equation}
        \mathcal{E}_1(\mathcal{L}, \mathcal{G}, \mathcal{P}, \mathcal{R}, g) \;\coloneq\; \frac{\|\mathcal{L}(g) - \mathcal{R}\circ\mathcal{G}\circ\mathcal{P}(g)\|_2}{\|\mathcal{L}(g)\|_2}.
    \end{equation}

    \item A sparsity penalty encourages the learned basis functions to have local support:

    \begin{equation}
        \mathcal{E}_2(\{\phi_j\}_{j=1}^N) \;\coloneq\; \frac{1}{N} \sum_{j=1}^N \|\phi_j\|_{L^1}.
    \end{equation}

    \item A reconstruction penalty encourages the learned subspace to accurately represent $k$-step iterates:
    \begin{equation}
        \mathcal{E}_3(\mathcal{L}, \mathcal{P}, \mathcal{R}, g; k) \;\coloneq\; \frac{\|\mathcal{L}^k(g) - \mathcal{R} \circ \mathcal{P}(\mathcal{L}^k(g))\|_2}{\|\mathcal{L}^k(g)\|_2},
    \end{equation}
    where $k \geq 1$. When $k=1$, this reduces to reconstruction of the single iteration of $\mathcal{L}$.
    
\end{enumerate}

The complete loss function is
\begin{equation}
    J(\mathcal{L}, \mathcal{G}, \mathcal{P}, \mathcal{R}, g, \{\phi_j\}_{j=1}^N)= \beta_1 \mathcal{E}_1 + \beta_2 \mathcal{E}_2 + \beta_3 \mathcal{E}_3,
\end{equation}
where $\beta_1, \beta_2, \beta_3\ge0$.

We also explored a supplementary projection penalty aimed at reconstructing the original observable:
\begin{equation*}
\mathcal{E}_{p_{1}}(\mathcal{P}, \mathcal{R}, g) \;\coloneq\; \frac{\|g - \mathcal{R} \circ \mathcal{P}(g)\|_2}{\|g\|_2}.
\end{equation*}
In practice, we found that this penalty sometimes adversely affected eigenpair reconstruction, likely by forcing the subspace to represent arbitrary inputs rather than  dynamically relevant directions.

We note that the inner-product structure of the projection block naturally encourages the learned basis to become approximately orthonormal during training. The composition $Q \coloneq \mathcal{R_{\phi}} \circ \mathcal{P}_{\phi}$ maps $g \mapsto \sum_{j=1}^N \langle g, \phi_j \rangle \phi_j$, which coincides with the orthogonal projection onto $\operatorname{span}\{\phi_j\}_{j=1}^{N}$ only when $\{\phi_j\}_{j=1}^{N}$ is orthonormal. For non-orthonormal bases, this formula incurs additional reconstruction error compared to the true optimal projection. Since $\mathcal{E}_1$ penalises reconstruction error, the optimiser is driven toward orthonormal configurations. As a side benefit, orthonormal bases improve the conditioning of the loss landscape: gradients with respect to different $\phi_j$ become more independent, reducing redundancy and facilitating convergence.

Finally, we discuss the activation functions employed in the proposed architecture. 
For the $\mathcal{G}$ network, which is learning the action of a linear operator (transfer or Koopman), we selected a linear activation function, namely the identity function, and set the bias term to zero. 
The activation functions for the encoder block ($\mathcal{E}$) can be chosen among the usual candidates, and we had the best results with ReLU in our experiments.

The combined architecture and loss in SABON are structured so that the learned basis $\{\phi_j\}_{j=1}^{N}$ (i) is approximately orthogonal, (ii) spans a subspace $V$ that is adapted to the geometry of the dynamics, leading to efficient spectral approximations, and (iii) has a tendency toward local support, guided by the sparsity term $\mathcal{E}_2$ in the loss function.
Further, $\mathcal{E}_1$ \textit{simultaneously optimises} for a good finite–dimensional approximation of the operator action and enforces that $\{\phi_j\}_{j=1}^{N}$ spans a space $V$ well suited for representing $\mathcal{L}(g)$. A final rescaling step yields an orthonormal basis.

\subsection{Universal approximation theorem}

Hua \textit{et al.} (2023) \cite{HUA202321} consider a general, possibly nonlinear, operator $\mathcal{J}$ between two distinct Hilbert spaces.

\begin{equation*}
\mathcal{J}\colon\mathcal{H}_\text{in}\longrightarrow\mathcal{H}_\text{out},
\qquad\mathcal{H}_\text{in}\neq\mathcal{H}_\text{out}
\end{equation*}
and therefore employ two autoencoders to learn separate finite‑dimensional subspaces for the domain and codomain.

Each of the transfer  and Koopman operators map $L^2(X)$ into $L^2(X)$, so only a single subspace $V\subset L^{2}(X)$ is required to approximate both the observable $g$ and its image $\mathcal{L}(g)$. Consequently, in contrast to \cite{HUA202321}, one should instead use a \textit{single} encoder. We state an approximation theorem in this alternate setting.

\begin{theorem}
\label{theorem_univeral_approximation}
Let $X \subset \mathbb{R}^d$ be compact. Let $\mathcal{L}: L^2(X) \rightarrow L^2(X)$ be a continuous linear operator. 
Given $N\in\mathbb N$ and a family $\phi=\{\phi_{1},\dots,\phi_{N}\}\subset L^{2}(X)$ define
\begin{align*}
&\mathcal P_{\phi}\colon L^{2}(X)\to\mathbb R^{N},&
\qquad
&\mathcal P_{\phi}(g) =\bigl(\langle g,\phi_{1}\rangle,\dots,\langle g,\phi_{N}\rangle\bigr)^{\!\top},\\
&\mathcal R_{\phi}\colon\mathbb R^{N}\to L^{2}(X),&
\qquad
&\mathcal R_{\phi}(a_{1},\dots,a_{N})=\sum_{j=1}^{N}a_{j}\phi_{j}.
\end{align*}
Then for each compact set $\mathcal{K} \subset L^2(X)$ and any $0 < \epsilon < 1$ there exist $N > 0$, a neural basis $\phi = \{\phi_1, \phi_2, \dots, \phi_N\}: X \rightarrow \mathbb{R}^N$ and a linear map $\mathcal{G}: \mathbb{R}^N \rightarrow \mathbb{R}^N$ such that the operator $\hat{\mathcal{L}} \coloneq \mathcal{R}_{\phi} \circ \mathcal{G} \circ \mathcal{P}_{\phi}$ satisfies
    \begin{equation*}
        \sup_{g \in \mathcal{K}} \left\| \mathcal{L}(g) - \hat{\mathcal{L}}(g) \right\|_{L^2(X)} \leq \epsilon.
    \end{equation*}
\end{theorem}

Although Theorem \ref{theorem_univeral_approximation} applies to any continuous linear operator $\mathcal{L}: L^2(X) \rightarrow L^2(X)$, we will focus on the transfer and Koopman operators.
\begin{remark}
\label{remark_theorem}
Examples of compact sets $\mathcal{K}\subset L^{2}(X)$ with $X\subset\mathbb{R}^{d}$, include (i) the closed unit ball of any finite-dimensional subspace in $L^{2}(X)$ and (ii) the closed unit ball of $H^{1}(X)$, for $d\ge 2$ and $X$ a bounded $C^1$ domain (its boundary is locally a $C^1$ graph).
The latter is compact in $L^2(X)$ by the Rellich--Kondrachov Embedding Theorem (e.g.\ Theorem 9.16 \cite{brezis}).

In the context of the Koopman and transfer operators associated with a dynamical system $T:X \rightarrow X$, two convenient sufficient conditions that guarantee the continuity of $\mathcal L$ are:
\begin{enumerate}
    \item $T$ is measurable and $\mu-$invariant, so that $\mathcal L$ acts as an isometry on $L^2(X)$. The circle rotation in \cref{subsec:circle_rotation} is a representative example with $\mu = \mathrm{Lebesgue}$;
    \item $T$ is non-singular with respect to $\mu$ and there exists a constant $c > 0$ such that
    \begin{equation*}
        0 < c^{-1} \le \dfrac{d(\mu\circ T^{-1})}{d\mu} \le c < \infty
    \end{equation*}
    $\mu$ almost everywhere. The Anosov map in \cref{subsec:nonlinearly_cat_map} is a representative example with $\mu=$ Lebesgue.
\end{enumerate}
\end{remark}

\begin{proof}[Proof of Theorem \ref{theorem_univeral_approximation}]
The proof proceeds almost exactly as in \cite[Theorem 2.2]{HUA202321}, so we omit the repeated details and focus on the necessary modifications.
We need only one neural basis $\{\phi_1,\ldots,\phi_N\}\subset L^2(X)$ (cf. \eqref{eq:vector_valued_neural}) for both the input and output spaces. From the universal approximation property \cite[Theorem\,2.1]{HUA202321}, the class of neural bases is dense in $L^2(X)$: for any compact $\mathcal{K} \subset L^2(X)$ and $\eta > 0$, there exist $N \in \mathbb{N}$ and a neural basis $\{\phi_1,\ldots,\phi_N\}\subset L^2(X)$ such that $\sup_{v \in \mathcal{K}} || v - Qv||_{L^2(X)} \leq \eta$, where $Q\coloneq \mathcal R_{\phi}\mathcal P_{\phi}$ is the projection onto $\operatorname{span}\{\phi_j\}_{j=1}^{N}$. Let $\mathcal{K}_{\mathcal{L}}:=\mathcal K\cup \mathcal L(\mathcal K)$, which is compact by continuity of $\mathcal{L}$, and choose $N$ and $\phi$ by \cite[Theorem 2.1]{HUA202321} so that,
\begin{equation*}
    \sup_{v\in \mathcal{K}_{\mathcal{L}}}\|v-Qv\|_{L^2(X)}\le \eta \quad \text{for some }\eta>0.
\end{equation*}
Next, since $Q$ is continuous and $\mathcal K$ is compact, $\mathcal{K}_{\mathcal{QL}}:=\mathcal{K}_{\mathcal{L}}\cup Q(\mathcal{K}_{\mathcal{L}})$ is compact, hence $\mathcal L$ is uniformly continuous on $\mathcal{K}_{\mathcal{QL}}$ with modulus of continuity $\omega_{\mathcal{L}}$.

The main difference appears in the error decomposition. In \cite{HUA202321}, one bounds three terms to handle potentially different neural bases for input and output spaces, but here we merge everything into two terms. If $\hat{\mathcal{L}} \coloneq \mathcal{R}_{\{\phi_j\}_{j=1}^{N}} \circ \mathcal{G} \circ \mathcal{P}_{\{\phi_j\}_{j=1}^{N}}$, then for $g\in \mathcal{K}$

{\small
\begin{equation*}
\begin{aligned}
\|\mathcal{L}(g) - \hat{\mathcal{L}}(g)\|_{L^2(X)} 
&\le\Bigl\|\mathcal{L}(g)- \mathcal{R}_{\{\phi_j\}_{j=1}^{N}} \circ \mathcal{P}_{\{\phi_j\}_{j=1}^{N}} \circ \mathcal{L} (g) \Bigr\|_{L^2(X)}\\
&\quad+\Bigl\|\mathcal{R}_{\{\phi_j\}_{j=1}^{N}} \circ \mathcal{P}_{\{\phi_j\}_{j=1}^{N}} \circ \mathcal{L} (g)-\mathcal{R}_{\{\phi_j\}_{j=1}^{N}} \circ \mathcal{G} \circ \mathcal{P}_{\{\phi_j\}_{j=1}^{N}}(g) \Bigr\|_{L^2(X)}.
\end{aligned}
\end{equation*}}

For the first term, since $\mathcal L(g)\in \mathcal L(\mathcal K)\subset \mathcal{K}_{\mathcal{L}}$, we have $\|\mathcal L(g)-Q\mathcal L(g)\|\le \eta$.
Setting $\mathcal G:=\mathcal P_\phi\circ \mathcal L\circ \mathcal R_\phi$ in the second term we obtain
\begin{equation*}
    \|Q\,\mathcal L(g)-Q\,\mathcal L(Qg)\|\le\|Q\|\,\|\mathcal L(g)-\mathcal L(Qg)\|
\le\|Q\|\,\omega_{\mathcal L}(\|g-Qg\|)\le\|Q\|\,\omega_{\mathcal L}(\eta).
\end{equation*}

Everything else remains unchanged from \cite[Theorem 2.2]{HUA202321}, and thus we conclude
\begin{equation*}
    \sup_{g\in \mathcal{K}}\|\mathcal{L}(g)-\hat{\mathcal{L}}(g)\|_{L^2(X)}
    \le \eta+\|Q\|\,\omega_{\mathcal L}(\eta).
\end{equation*}
Choosing $N$ and $\phi$ so that the right-hand side is smaller than $\epsilon$ yields the claim.
\end{proof}

\section{Numerical examples}
\label{sec:examples}

In this section, we illustrate the results obtained with our architecture through three representative examples of increasing complexity. We consider the Perron--Frobenius operator in all examples, but the procedure is identical for training data generated by the Koopman operator.

The first example considers a circle rotation, highlighting the recovery of meaningful eigenvalues and eigenfunctions of the transfer operator. The remaining two examples examine nonlinear Anosov diffeomorphisms on the two-torus: a weakly nonlinear perturbation of Arnold's cat map, and a strongly nonlinear variant. These hyperbolic maps are highly anisotropic and the eigenfunctions of their transfer operators are known to be \emph{distributions} rather than functions \cite{GL06}. Through these challenging examples, we demonstrate how the learned basis functions adapt to the system dynamics, providing efficient and accurate recovery of spectral properties such as the SRB measure.

\subsection{Data generation}
\label{subsec:data_generation}
All examples in this paper --- the circle rotation and the two Anosov maps --- are dynamical systems on periodic domains.
We therefore require periodic observables, and we use random trigonometric polynomials to provide a rich class of training functions. To train the proposed architecture, we consider the intrinsic dimension $d' = \dim \; X$ and generate a dataset of $D$ random trigonometric polynomial functions defined on a $d'$-dimensional domain. Each such function is constructed as a finite linear combination of sinusoidal basis functions (sine and cosine terms) in each dimension, with randomly chosen coefficients. Formally, a general $d'$-dimensional trigonometric polynomial of maximum order $K$ can be written as a linear combination of terms of the form
$\prod_{j=1}^{d'} \Psi_j(k_j x^{(j)})$,
where each $\Psi_j(\cdot)$ is either $\sin(\cdot)$ or $\cos(\cdot)$, and each integer frequency $k_j$ ranges from $1$ to $K$ for $\sin(\cdot)$ and from $0$ to $K$ for $\cos(\cdot)$. In other words, no oscillatory component in any dimension exceeds frequency $K$. The coefficients in the linear combination are real numbers drawn at random uniformly in $[-1,1]$.
Note that the inclusion of $k_j = 0$ allows constant (zero-frequency) components in each dimension as well.
For example, in $d'=2$ dimensions (with coordinates $x^{(1)}$ and $x^{(2)}$), such trigonometric polynomials are linear combinations of
$\sin(k_1 x^{(1)})\sin(k_2 x^{(2)}), \sin(k_1 x^{(1)})\cos(k_2 x^{(2)}), \cos(k_1 x^{(1)})\sin(k_2 x^{(2)})$, and $\cos(k_1 x^{(1)})\cos(k_2 x^{(2)}),$
 with $k_1, k_2 \in \{1,\ldots,K\}$ for the sine terms and $k_1, k_2\in \{0,1,\ldots,K\}$ for the cosine terms. 
In the general $d'$-dimensional case, each term is a product of trigonometric functions, one for each of the $d'$ coordinates, with frequencies up to $K$ in each coordinate.
Such functions contain a rich variety of mixed-frequency components across all dimensions.

 Once the $D$ random functions ${f_j}$ are generated as described above, we apply $\mathcal{L}$ to each of them to obtain the corresponding output functions $\mathcal{L}f_j$. This yields a training set of $D$ input-output pairs $\{f_j;\mathcal{L}f_j\}_{j=1}^D$. If the reconstruction loss $\mathcal{E}_3$ is used with index $k > 1$, we additionally use the $k$-step iterates $\mathcal{L}^k f_j$. In our training, $[f_j(x_1),f_j(x_2),\ldots,f_j(x_n)]^{\top}$ serves as an input function evaluated on the discretization $\{x_i\}_{i=1, \dots, n}$, and $[\mathcal{L}f_j(x_1),\allowbreak\mathcal{L}f_j(x_2),\ldots,\mathcal{L}f_j(x_n)]^{\top}$ is treated as the target output. By using trigonometric polynomials as input functions, we ensure that the action of $\mathcal{L}$ is well defined and smooth on these inputs. This property simplifies the generation of training data when $\mathcal{L}f_j$ can be computed analytically, and preserves the continuity that is important for machine-learned approximations.
 
 In summary, our choice of trigonometric polynomials is natural because (i) they form a very expressive class of functions, which are dense in the space of continuous and square-integrable periodic functions, (ii) they are smooth, facilitating stable training of the network (gradients can be computed reliably, and there are no discontinuities or singularities in the input data), and (iii) their generation is straightforward and computationally inexpensive.

\subsection{Examples setup}
We train on collections of functions as described in \cref{subsec:data_generation}. In our examples, the discretizations $\{x_i\}_{i=1, \dots, n}$ lie on a uniform grid in the domain, but we could also have chosen a discretization based on scattered points.
The validation and test sets are drawn from the same underlying class of functions as our training data.
The parameters of each dataset are summarised in \cref{tab:data_specifics}.
\begin{table}[H]
\caption{Data parameters for the circle rotation and perturbed cat map. We list the number of training functions ($D$), maximum trigonometric polynomial order ($K$), number of validation functions, number of test functions, and the grid size ($n$).}
\label{tab:data_specifics}
\centering
\begin{tabular}{|l|l|l|l|l|l|}
\hline
Example           & $D$ & $K$ & Validation & Test & $n$ \\ \hline
Circle rotation   & 1000  & 9          & 500        & 100  & 100       \\ \hline
Perturbed cat map & 3000  & 5          & 500        & 500  & 10000         \\ \hline
\end{tabular}
\end{table}

The training set is used to fit the model parameters by minimising the defined loss.
During training, we periodically evaluate the current model on the validation set. We also explore alternative hyperparameter settings --- e.g.\ learning rate, network depth, and regularisation strength—and retain the configuration that achieves the lowest validation error.
Only after all training and hyperparameter selection are complete do we assess the model on the test set. This provides an unbiased estimate of the model's performance on unseen functions, since the test data is not used to guide model selection.

For more complex dynamics, like the nonlinearly perturbed cat map, we employ larger architectures to capture the increased complexity. The specifics of each architecture, including the dimensions, activation functions, and the number of learned basis functions are shown in \cref{tab:architecture_specifics}.

\begin{table}[H]
\caption{Architecture specifics for the circle rotation and nonlinear cat map. We list the encoder ($\mathcal{E}$) hidden layer architecture (e.g.\ $5\times512$ denotes five hidden layers with 512 units each), activation functions, total parameter count, and training time. The approximation network $\mathcal{G}$ is a linear map (i.e.\ a matrix) acting on $\mathbb{R}^N$ for all examples.}
\label{tab:architecture_specifics}
\centering
\begin{tabular}{|l|ll|l|r|r|}
\hline
\multirow{2}{*}{Example} & \multicolumn{2}{l|}{Encoder ($\mathcal{E}$)} & \multirow{2}{*}{$N$} & \multirow{2}{*}{Params.} & \multirow{2}{*}{Time} \\ \cline{2-3}
 & \multicolumn{1}{l|}{Layers} & Act. & & & \\ \hline
Circle rotation & \multicolumn{1}{l|}{5 $\times$ 512} & ReLU & 19 & 1.06M & 2.25 min \\ \hline
Nonlinear cat map & \multicolumn{1}{l|}{5 $\times$ 2048} & ReLU & 324 & 17.6M & 17.65 min \\ \hline
Conj.\ nonlinear cat map & \multicolumn{1}{l|}{5 $\times$ 2048} & ReLU & 676 & 18.6M & 37.13 min \\ \hline
\end{tabular}
\end{table}

All experiments were performed on an A100 NVIDIA GPU, and the model was implemented using PyTorch. We trained for 10,000 epochs on the circle rotation and 4,500 epochs on each Anosov map. We briefly comment on sensitivity to hyperparameters. As shown in \cref{tab:architecture_specifics}, more complex dynamics may require a network with slightly larger capacity (e.g.\ wider hidden layers). However, the dominant factor affecting performance is the number of basis functions $N$, which must grow with the complexity of the dynamics to adequately capture the operator dynamics. Regarding the multi-step reconstruction loss $\mathcal{E}_3$, we observe that the choice of index $k$ should reflect the anisotropy of the dynamics. For systems with more irregular structures in certain directions --- such as Anosov diffeomorphisms, whose eigendistributions are smooth along unstable manifolds but highly irregular along stable ones --- larger values of $k$ are beneficial. This is because higher iterates $\mathcal{L}^k(g)$ increasingly expose the anisotropic geometry of the dynamics, encouraging the learned basis to adapt to these irregular directions. Across our three examples, we used $k=0$ for the circle rotation and $k=2$ for the two Anosov maps.

\subsection{Transfer operator for a circle rotation}
\label{subsec:circle_rotation}
The circle rotation map $T:S^1 \to S^1$ is a transformation in the circle given by $T(\theta)= \theta + \alpha\pmod{2\pi}$.
For our numerical examples, we choose $\alpha=-1$. 
To avoid discontinuities at the endpoints of the domain interval, we embed the unit-radius circle $S^1$ into $\mathbb{R}^2$, identifying each $\theta$ with $(\cos \theta,\sin \theta)\in \mathbb{R}^2$. Thus the ambient dimension is $d=2$ while the intrinsic dimension remains $d'=1$.

\begin{figure}[H]
  \centering
  \includegraphics[width=1\linewidth]{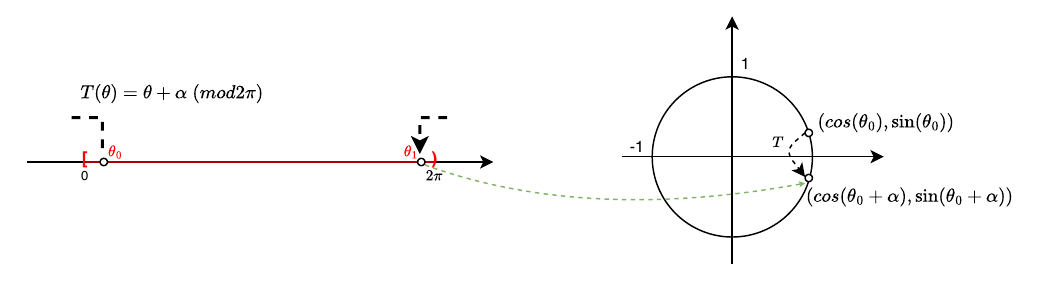}
  \caption{Embedding of the circle‑rotation map
$T(\theta)=\theta+\alpha \pmod{2\pi}$ from the interval
$[0,2\pi)$ into the unit circle $S^{1}\subset\mathbb{R}^{2}$.
Each angle $\theta$ is sent to its Cartesian coordinates
$(\cos\theta,\sin\theta)$, removing the discontinuity
that would occur at $\theta=0\equiv 2\pi$ in the interval picture.
The initial point $\theta_{0}$ is chosen close to the right‑hand
endpoint of the interval to highlight the jump that is resolved by
the embedding.}
\label{fig:circle_rotation_embedding}
\end{figure}

After the learning routine, we obtain both a set of basis functions spanning a subspace well-suited for spectral approximation and an estimator $\hat{\mathcal{L}}$ for the transfer operator's action.
\cref{fig:circle_rotation_input_output} displays how $\hat{\mathcal{L}}$ performs on unseen data for various sparsity scenarios.
\begin{figure}[H]
  \centering
\includegraphics[width=0.92\linewidth]{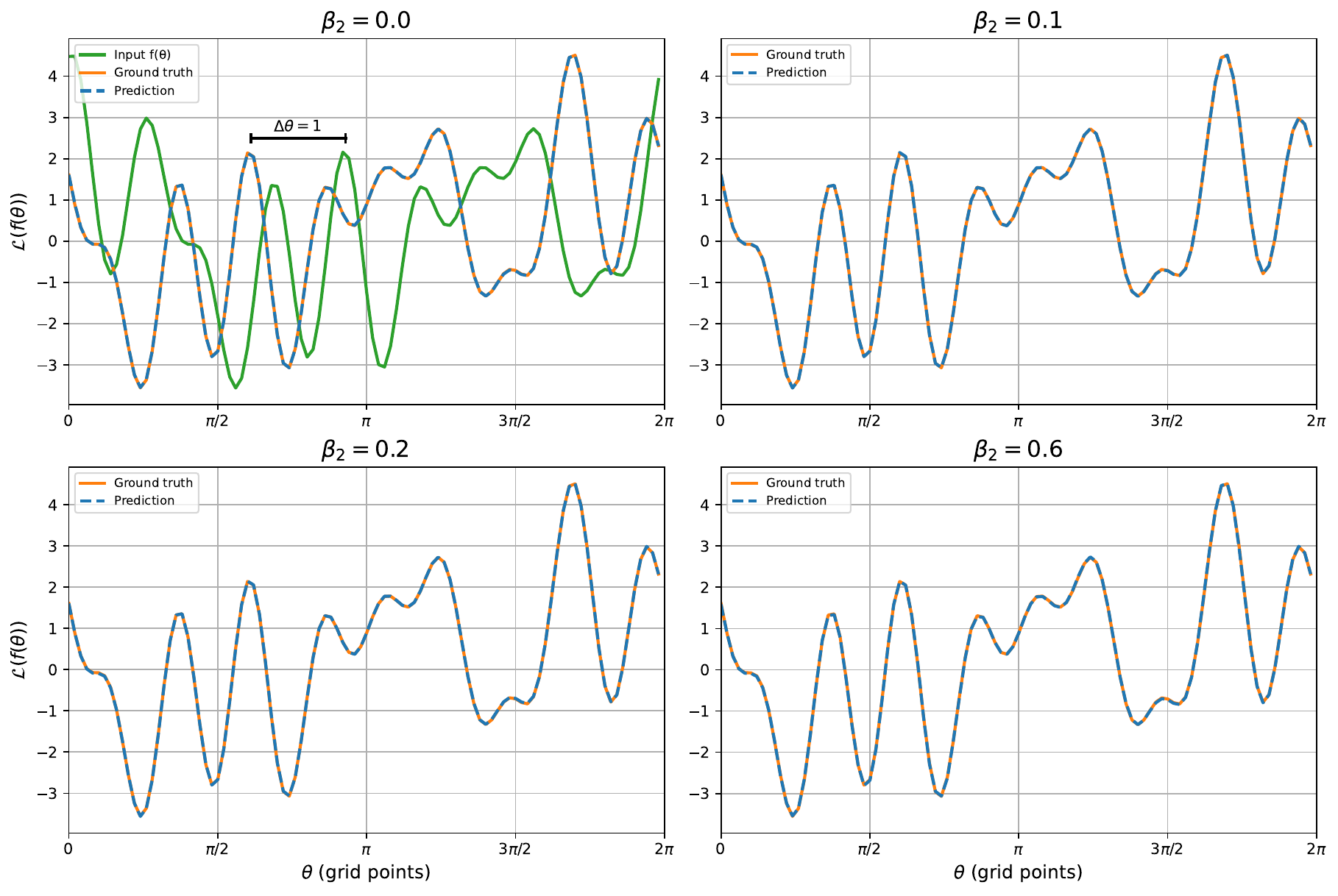}
  \caption{Performance of the learned operator approximation on unseen input data, for different levels of sparsity. No appreciable difference in performance is seen across the various sparsity penalties.}
\label{fig:circle_rotation_input_output}
\end{figure}
The corresponding quantitative results --- mean relative $L^2$ errors averaged over the test set --- are summarised in \cref{tab:mean_rel_errors}.
\cref{tab:mean_rel_errors} is consistent with \cref{fig:circle_rotation_input_output}; there is little change in relative error observed with increasing sparsity.
\begin{table}[H]
\caption{Mean relative $L^2$ error of each model on the test set.}
\label{tab:mean_rel_errors}
\centering
\begin{tabular}{|l|c|}
\hline
Model & Mean relative error \\ \hline
No Sparsity ($\beta_2 = 0$) & $3.974\times 10^{-3}$ \\ \hline
Sparsity ($\beta_2 = 0.1$) & $3.641\times 10^{-3}$ \\ \hline
Sparsity ($\beta_2 = 0.2$) & $3.958\times 10^{-3}$ \\ \hline
Sparsity ($\beta_2 = 0.6$) & $6.234\times 10^{-3}$ \\ \hline
\end{tabular}
\end{table}

The learned subspace is spanned by the set of basis functions depicted in \cref{fig:circle_rotation_basis_functions}. By introducing a sparsity penalty, these functions gain increasingly localised support, as shown in \cref{fig:circle_rotation_basis_functions}.
\begin{figure}[H] \centering \includegraphics[width=1\linewidth]{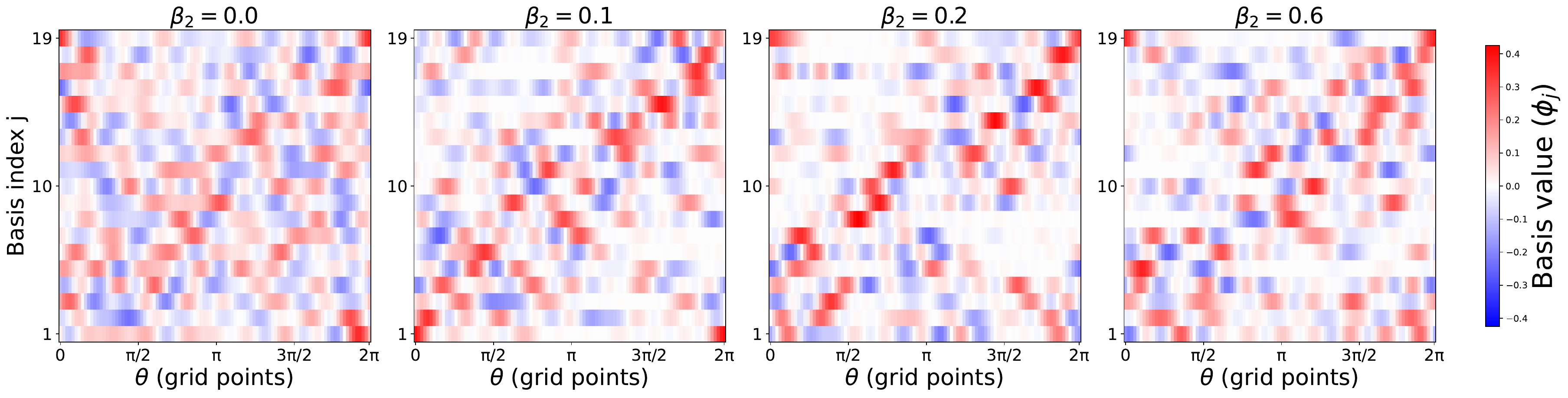} \caption{$L^2$-normalised basis functions obtained under different sparsity penalties: each panel displays a distinct basis for a distinct sparsity penalty. 
In a single panel, the state space coordinates (the angles $\theta\in S^1$) are along the $x$-axis, and each row shows the values of a different basis function. 
Higher sparsity encourages localisation, which is reflected in more concentrated regions of support.
} \label{fig:circle_rotation_basis_functions} \end{figure}

After $L^2$-normalising the learned basis, we examine the Gram matrix, defined as
\begin{equation}
\label{eq:gram_matrix}
M_{kj} = \langle \phi_k, \phi_j\rangle,
\qquad k,j = 1,\dots,N,
\end{equation}
in \cref{fig:circle_rotation_gram_matrix}. The introduction of a sparsity penalty slightly improves orthogonality.

\begin{figure}[H] \centering \includegraphics[width=1\linewidth]{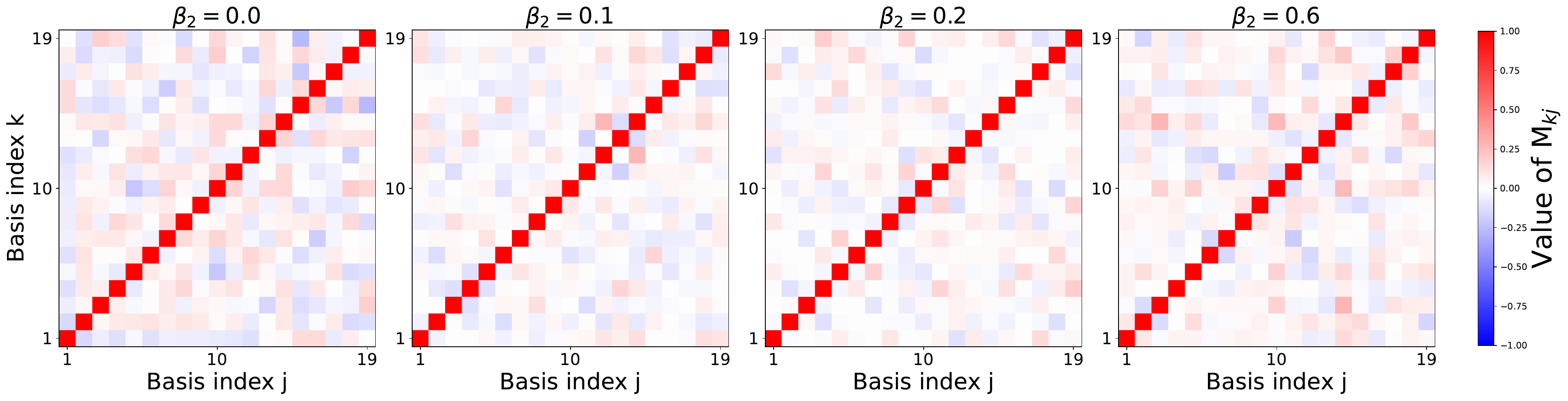} \caption{Gram matrix of the normalised basis functions}
\label{fig:circle_rotation_gram_matrix} \end{figure}

We now recover the spectral properties of $\mathcal{L}$ from the learned estimator $\hat{\mathcal{L}}$. 
Consider the collection of learned basis functions $\{\phi_j\}_{j=1}^{N} \subset L^2(X)$. Suppose an observable $g \in L^2(X)$ can be represented as
\begin{equation*}
    g = \sum_{j=1}^{N} \xi_j \phi_j,
\end{equation*}
where $\xi = (\xi_1, \dots, \xi_N)^T \in \mathbb{R}^N$ is the coefficient vector.

Applying the projection operator $\mathcal{P}_\phi$ to $g$ yields
\begin{equation*}
    \mathcal{P}_\phi(g) = \begin{pmatrix} \langle g, \phi_1 \rangle \\ \vdots \\ \langle g, \phi_N \rangle \end{pmatrix} = \begin{pmatrix} \langle \sum_{j=1}^N \xi_j \phi_j, \phi_1 \rangle \\ \vdots \\ \langle \sum_{j=1}^N \xi_j \phi_j, \phi_N \rangle \end{pmatrix} = M\xi,
\end{equation*}
where $M$ is the Gram matrix as defined in \eqref{eq:gram_matrix}.

The learned operator is $\hat{\mathcal{L}} = \mathcal{R}_\phi \circ \mathcal{G} \circ \mathcal{P}_\phi$ where $\mathcal{G}$ is an $N \times N$ matrix representing the dynamics in coefficient space. When $\hat{\mathcal{L}}$ is applied to $g$ we obtain the following linear transformation:
\begin{equation*}
    \hat{\mathcal{L}}g = \mathcal{R}_\phi(\mathcal{G}(\mathcal{P}_\phi(g))) = \mathcal{R}_\phi(\mathcal{G}(M\xi)) = \sum_{j=1}^{N} (\mathcal{G}M\xi)_j \phi_j.
\end{equation*}

A function $g$ with coefficient vector $\xi$ is an eigenfunction of $\hat{\mathcal{L}}$ with eigenvalue $\lambda$ if and only if $\hat{\mathcal{L}}g = \lambda g$. This means
\begin{equation*}
    \sum_{j=1}^{N} (\mathcal{G}M\xi)_j \phi_j = \lambda \sum_{j=1}^{N} \xi_j \phi_j.
\end{equation*}
Since the basis functions are approximately linearly independent, we must have
\begin{equation}
\label{eq:eigenproblem}
\mathcal{G}M\xi = \lambda \xi.
\end{equation}

Equation \eqref{eq:eigenproblem} is an eigenproblem of size $N\times N$. Solving it yields eigenvalue-eigenvector pairs $(\lambda_k, \xi^{(k)})$
for $k=1, \dots, N$. The corresponding eigenfunctions are
\begin{equation*}
    g_k = \sum_{j=1}^{N} \xi_j^{(k)} \phi_j, \quad k = 1, \dots, N.
\end{equation*}

For the circle rotation, one can derive the eigenpairs analytically.
The transfer operator for the circle rotation is
\begin{equation*}
    (\mathcal{L}g)(\theta) = g(\theta - \alpha).
\end{equation*}
Using the eigenfunction ansatz $\psi_k(\theta) = e^{i k\theta}$ for $k\in\mathbb{Z}$, we have
\begin{equation*}
    (\mathcal{L}\,\psi_k)(\theta) = e^{i k(\theta - \alpha)} 
= e^{-i k\alpha}\,\psi_k(\theta)
\quad\Longrightarrow\quad
\lambda_k = e^{-i k\alpha},\quad 
\psi_k(\theta) = e^{i k\theta},
\end{equation*}
which shows that we have eigenpairs $(\lambda_k,\psi_k), k\in\mathbb{Z}$. Furthermore, for an irrational rotation, these are the only eigenpairs: the Fourier modes form an orthonormal basis for $L^2(S^1)$, and the $2\pi$-irrationality of $\alpha$ ensures that all eigenvalues $\lambda_k = e^{-ik\alpha}$ are distinct.
Thus, the leading eigenvalues take the form
\begin{equation*}
    \lambda_{\pm1}=e^{\pm i\alpha}.
\end{equation*}
These analytic expressions match well with our numerical experiments, which were performed with the model parameter $\beta_{2}=0.6$; see \cref{fig:circle_rotation_eigenpairs}.
Furthermore, the leading learned eigenfunctions coincide perfectly with the leading analytic eigenfunctions; see \cref{fig:circle_rotation_eigenfunctions}.
\begin{figure}[H]
    \centering
    \includegraphics[width=0.91\linewidth]{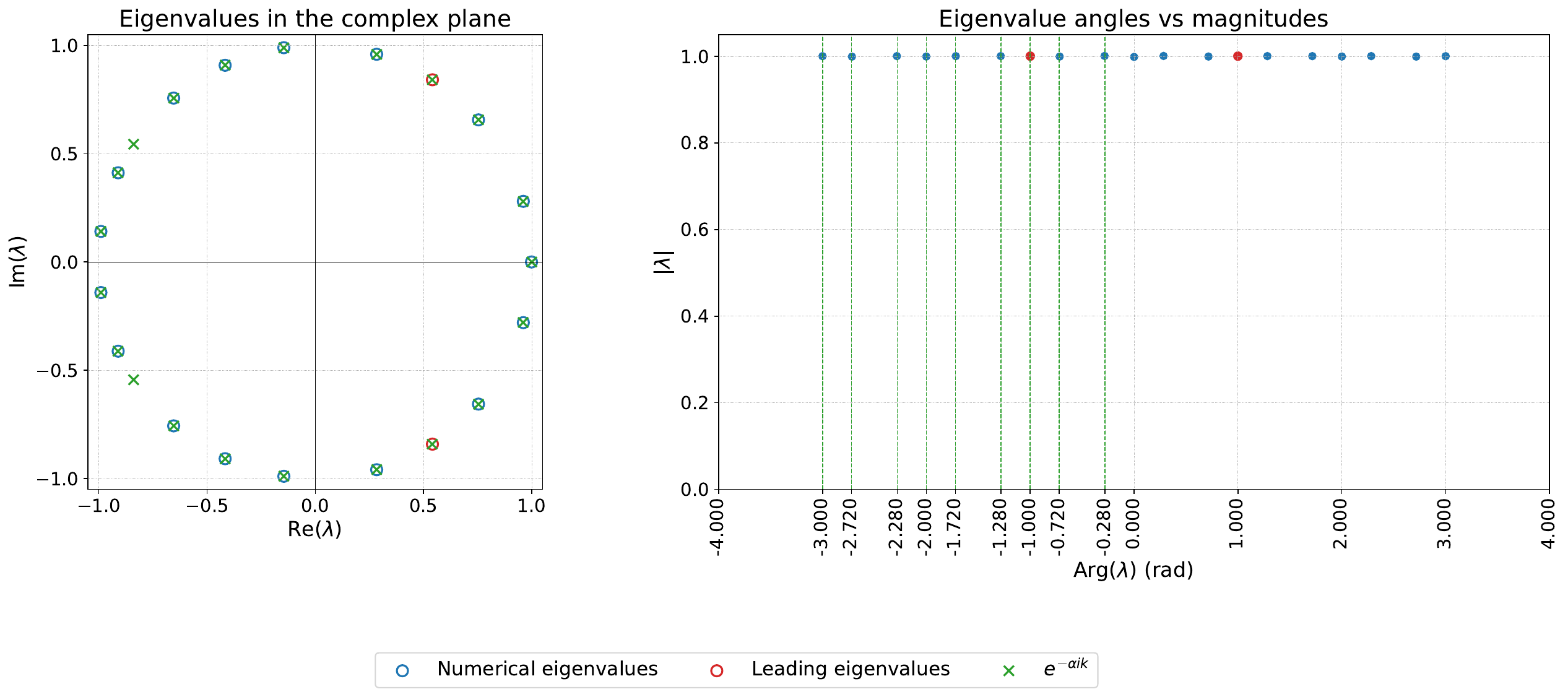}
    \caption{
        \textbf{Left:} Eigenvalues of the Perron–Frobenius operator for a circle rotation by $\alpha=-1$ radians. Numerical eigenvalues (blue circles; leading pair corresponding to $k=\pm 1$ shown as red circles) coincide with the analytic spectrum (green crosses) 
        \textbf{Right:} Magnitude versus argument of the same eigenvalues. The dashed green vertical lines mark angles that are negative integer multiples of $\alpha$, underscoring that the additional numerical eigenvalues are dynamically meaningful, but duplicate information from the lowest-order eigenvalues.}
    \label{fig:circle_rotation_eigenpairs}
\end{figure}

\begin{figure}[H]
    \centering
\includegraphics[width=0.75\linewidth]{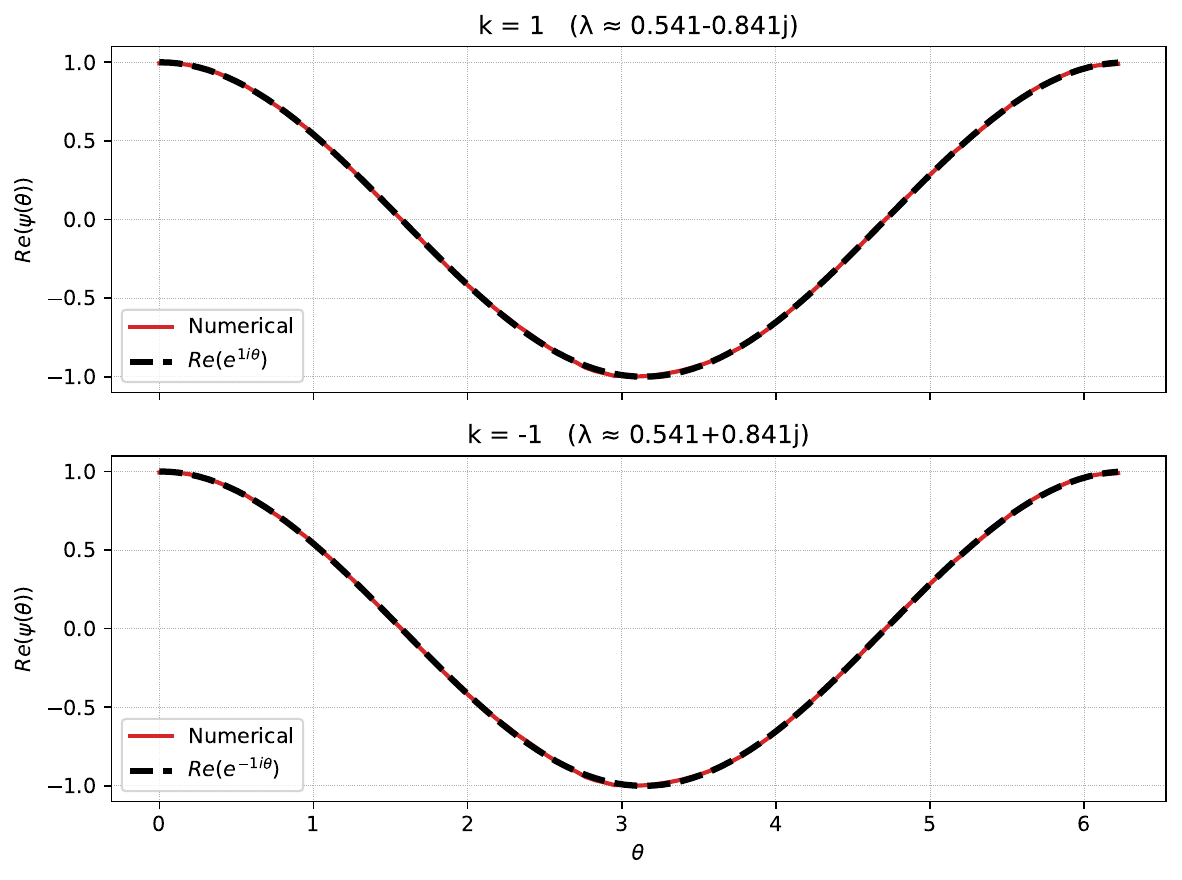}
    \caption{Real parts of the leading eigenfunctions of the transfer operator. The analytic eigenfunction is plotted in dashed black, while its numerical counterpart is in solid red. In both cases the numerical eigenfunctions coincide with the analytic curves.}
    \label{fig:circle_rotation_eigenfunctions}
\end{figure}

We close this example by discussing the choice of $N=19$ basis functions.
Recall we trained $\hat{\mathcal{L}}$ using trigonometric polynomials up to order 9 (equivalently linear combinations of real Fourier modes up to order 9).
As discussed above, each two-dimensional real subspace spanned by $\theta\mapsto\sin(k\theta)$ and $\theta\mapsto\cos(k\theta)$ is an eigenspace with eigenpair $\lambda_{\pm k}$. 
Thus, in this special situation of circle rotation dynamics, the input and output data lie in the same 19-dimensional ($19=2\times 9 + 1$) subspace, which is itself invariant under $\mathcal{L}$. Consequently, SABON learns an invariant subspace in this example. Increasing the learned basis cardinality $N$ beyond 19 will not impact the prediction error, but the additional basis dimensions cannot easily be controlled with $9^{\rm th}$-order input/output data, and the resulting subspace will no longer be invariant.

\subsection{Transfer operator for two nonlinear Anosov maps}
\label{subsec:anosov_maps}

We now examine two Anosov diffeomorphisms on the two-torus $\mathbb{T}^2$, both possessing well-defined SRB measures. The first is a weakly nonlinear perturbation of Arnold's cat map, where the stable and unstable directions remain close to those of the original map. The second is obtained by conjugating the perturbed with a nonlinear diffeomorphism, yielding a strongly nonlinear map, which is far from volume preserving and whose stable and unstable directions vary significantly over the phase space

This pair of examples allows us to investigate how SABON performs as the geometric complexity of the dynamics increases. The strongly nonlinear conjugated map presents a more challenging test case, as the basis functions must adapt to curved unstable foliations and greater variation in local volume expansion and contraction.

\subsubsection{A nonlinearly perturbed Arnold's cat map}
\label{subsec:nonlinearly_cat_map}
We consider a nonlinear perturbation of the (linear and hyperbolic) Arnold's cat map of the two-torus $X=\mathbb{T}^2$ introduced in \cite{CF20}:
\begin{equation*}
    T:\mathbb{T}^{2} \to \mathbb{T}^{2}, \qquad T(x,y)=\bigl(2x+y+2\delta\cos(2\pi x),x+y+\delta\sin(4\pi y+1)\bigr)\pmod{1},
\end{equation*}
with $\delta=0.01$.
Under this perturbation, $T$ remains Anosov and therefore has an SRB measure $\mu_{\mathrm{SRB}}$ \cite[Theorem 4.12]{bowen1975}.
The spectral theory for Anosov diffeomorphisms is challenging and relies on carefully constructed anisotropic Banach spaces (e.g.\ \cite{GL06}) that are adapted to the unstable and stable foliations.
In this example, we let $\mathcal{L}f=f\circ T^{-1}/|\det DT\circ T^{-1}|$ be the Perron--Frobenius operator, and consider $\mathcal{L}:L^2(X,m)\to L^2(X,m)$ where $m$ is Lebesgue measure on the flat torus.
Note the additional determinant term in the definition of $\mathcal{L}$; the constant function is no longer the leading eigenfunction of $\mathcal{L}$. 
In fact, the object $\mu_{\mathrm{SRB}}$ satisfying $\mathcal{L}\mu_{\mathrm{SRB}} = \mu_{\mathrm{SRB}}$ is in general a \emph{distribution} and need not lie in $L^2(\mathbb{T}^2,m)$. 
The fixed distribution $\mu_{\mathrm{SRB}}$ may be identified with the SRB measure and describes the long-term distribution of infinitely long trajectories of $T$ for \emph{Lebesgue almost all} initial conditions $x\in\mathbb{T}^2$.
This distribution of infinitely long trajectories is accessible through the fixed point of $\mathcal{L}$; this is one of the many advantages of studying transfer operators.

To eliminate the discontinuities that arise at the edges of the unit
square, we embed $\mathbb{T}^{2}$ into $\mathbb{R}^4$ by mapping each $(x,y)$ to its pair of complex phases $(\cos 2\pi x,\allowbreak \sin 2\pi x,\allowbreak \cos 2\pi y,\allowbreak \sin 2\pi y)$. That is, we treat a problem with intrinsic dimension $d'=2$ in a higher ambient space with dimension $d=4$.
We first evaluate our model’s ability to approximate the
Perron–Frobenius operator on an unseen (not in the training set) observable.
To highlight the anisotropic nature of the learned basis everywhere in the phase space, we chose not to apply the sparsity penalty throughout both Anosov examples. \cref{fig:anosov_input_groundtruth} displays an unseen input observable~$g$, the ground-truth image, and our network’s
prediction. The accuracy is quantified by the relative $L^{2}$ error of
$1.625\times10^{-2}$ on the full test set.

\begin{figure}[H] \centering \includegraphics[width=1\linewidth]{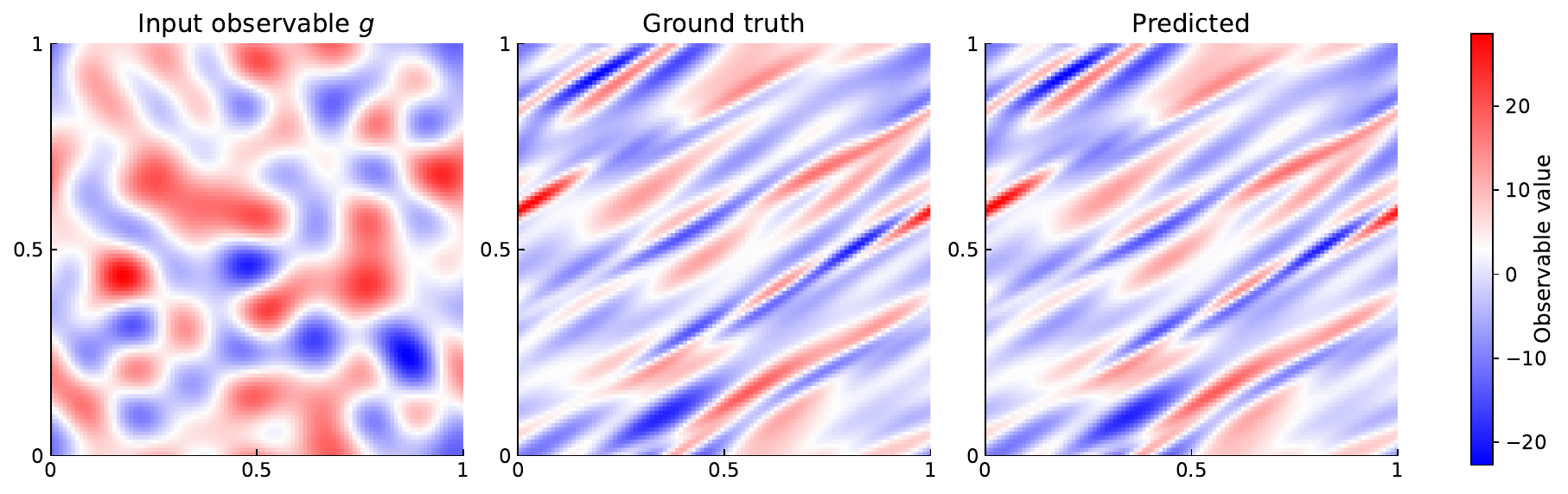} \caption{Input and output for the Perron–Frobenius operator of a perturbed Arnold’s cat map.
Left to right, the three panels show (i) an unseen input observable $g$, (ii) the corresponding ground-truth output from the true transfer operator, (iii) and the corresponding predicted output from $\hat{\mathcal{L}}$.
The mean $L^{2}$ relative error over the full test set is $1.625\times10^{-2}$.}\label{fig:anosov_input_groundtruth} \end{figure}

Next, we analyse the learned set of basis functions $\{\phi_j\}_{j=1}^{N}$ that span the subspace $V$. \cref{fig:anosov_basis_functions} shows representative basis functions obtained without sparsity penalty. The basis functions obtained through the proposed network show unmistakable alignment with contracting and expanding directions of the nonlinear dynamics generated by $T$.

\begin{figure}[H] \centering \includegraphics[width=1\linewidth]{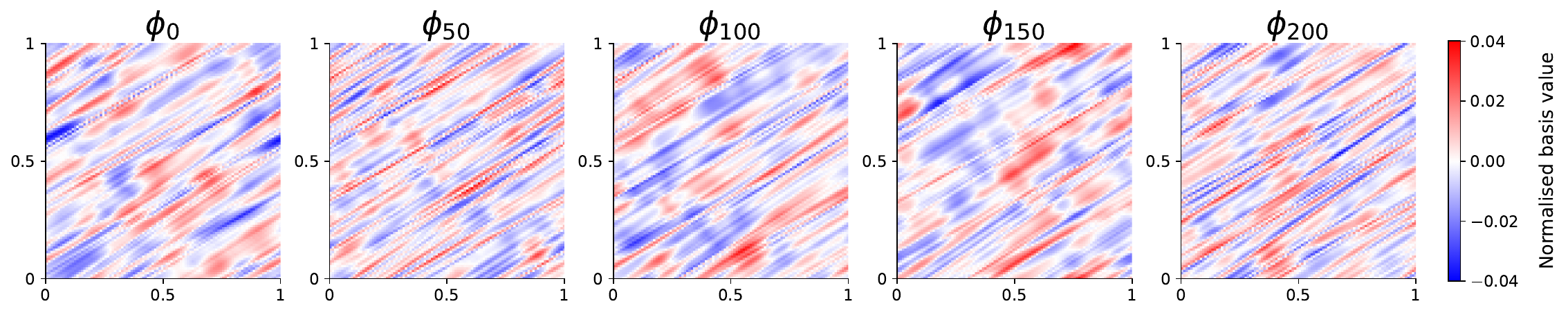} \caption{Representative basis functions learned by the network. The patterns exhibit structures that align with the map’s contracting and expanding directions, illustrating how the learned subspace is adapted to the underlying dynamics.} \label{fig:anosov_basis_functions} \end{figure}

The Gram matrix of the $L^2$-normalised learned basis functions defined in \eqref{eq:gram_matrix} is illustrated in \cref{fig:anosov_gram_matrix}.
Its nearly diagonal form shows that our training scheme implicitly promotes approximate orthogonality, even without an explicit sparsity term.

\begin{figure}[H]
    \centering
    \includegraphics[width=0.65\linewidth]{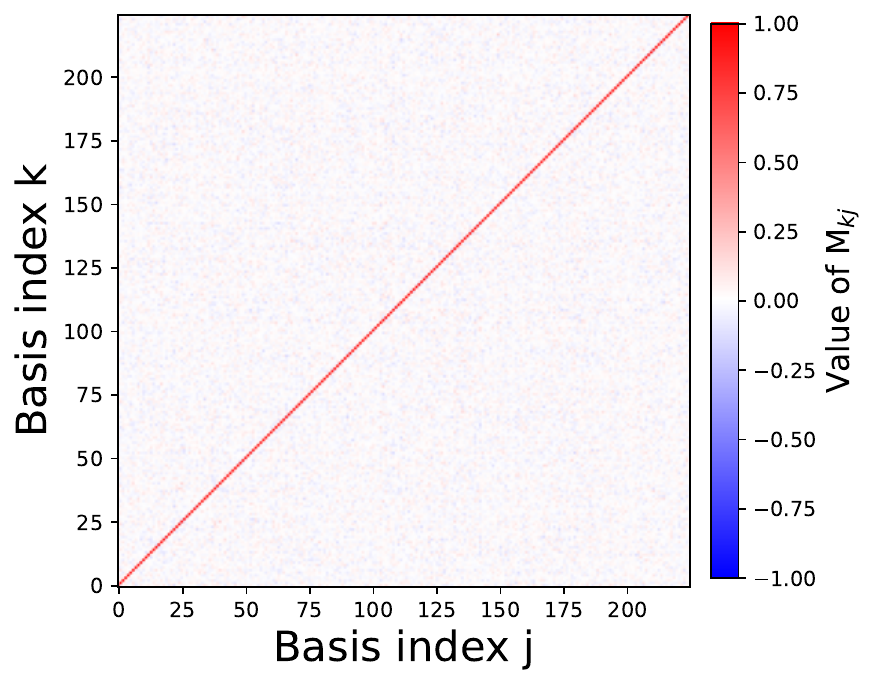}
    \caption{Gram matrix of the learned basis functions (after $L^2$ normalisation). The diagonal structure shows that the network has produced a nearly orthogonal collection, despite the absence of an explicit sparsity constraint.}
    \label{fig:anosov_gram_matrix}
\end{figure}

We now examine the spectral properties of the learned Perron–Frobenius approximator. \cref{fig:anosov_eigenvalues} displays the eigenvalues; the leading eigenvalue at $\lambda=1$ confirms the presence of an invariant measure as a fixed point of the Perron--Frobenius operator.
\begin{figure}[H] \centering \includegraphics[width=0.8\linewidth]{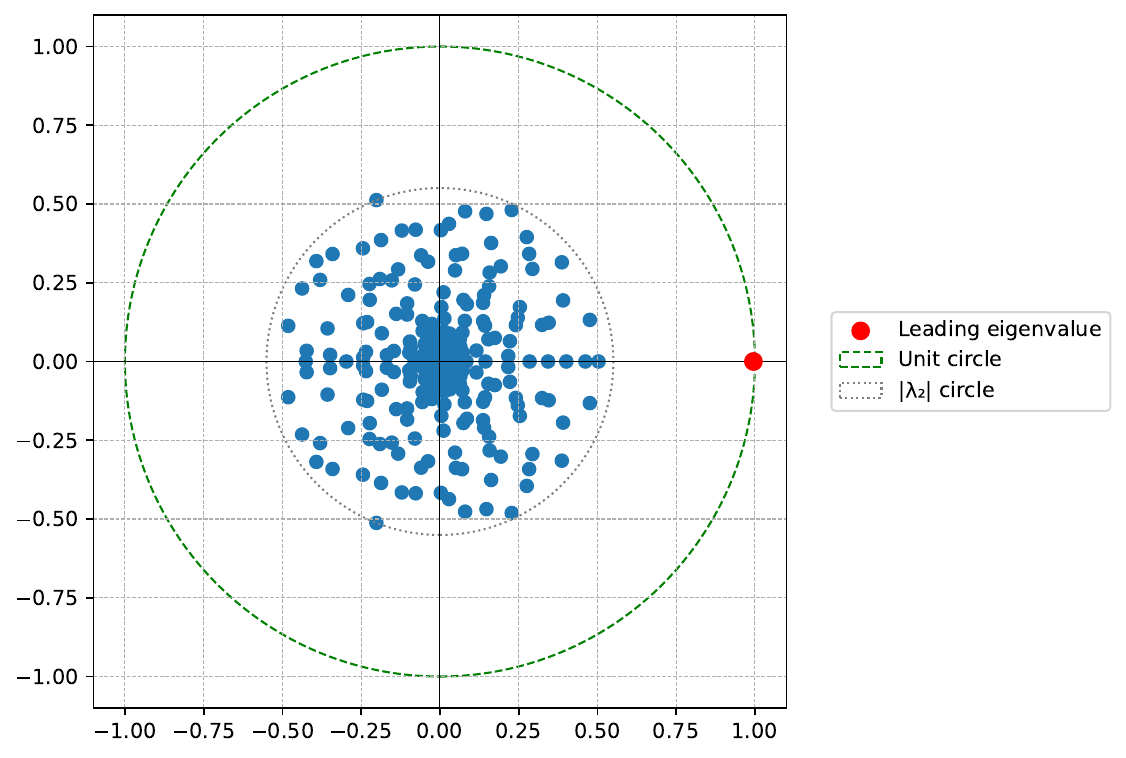} \caption{Eigenvalues of $\hat{\mathcal{L}}$ for the model with no sparsity penalty. The leading eigenvalue (corresponding to the SRB measure) is close to 1 as expected. The remaining eigenvalues are separated away from the unit circle.} \label{fig:anosov_eigenvalues} \end{figure}
The remaining eigenvalues in \cref{fig:anosov_eigenvalues} are restricted to a disk away from the unit circle, in line with expectations if one used a suitable anisotropic Banach space for the domain of $\mathcal{L}$ \cite{GL06}. 
While we work in $L^2(X,m)$, we train our estimate of $\mathcal{L}$ on smooth functions, and it is likely that this restriction to smooth functions produces the spectral gap we see in \cref{fig:anosov_eigenvalues}.

Solving the eigenequation \eqref{eq:eigenproblem} similarly to the previous example, we reconstruct the leading eigenfunction (corresponding to $\lambda=1$), which represents the SRB measure; see \cref{fig:srb_plot}.
\begin{figure}[H]
    \centering
    \includegraphics[width=1\linewidth]{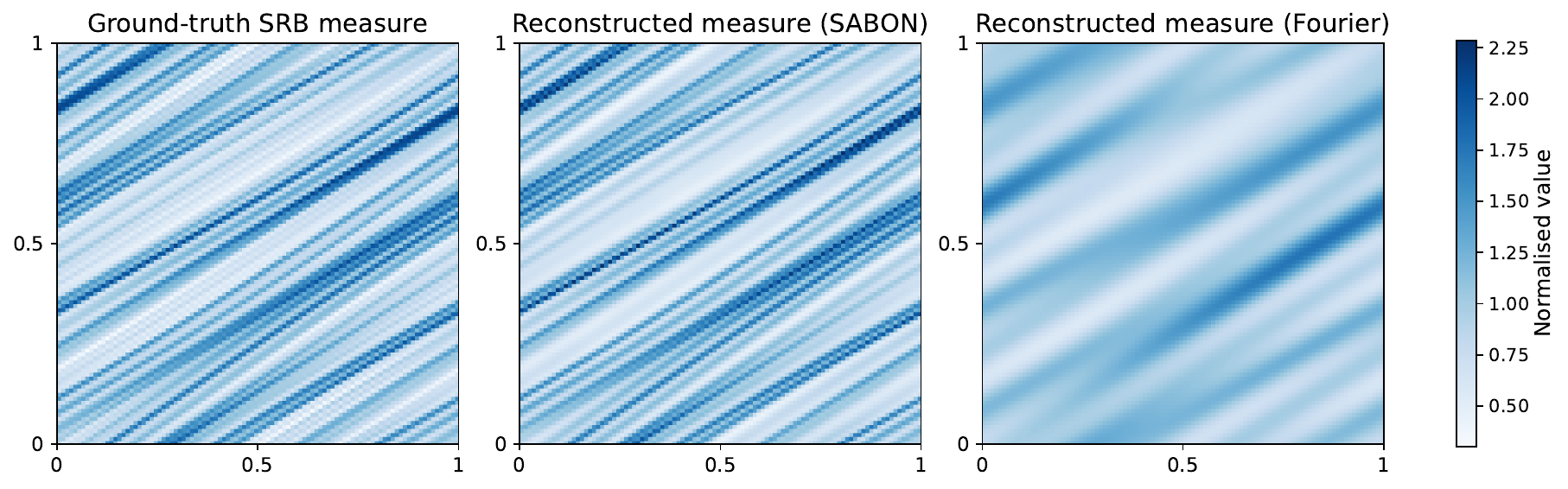}
\caption{Densities of ground-truth SRB measure (left), its approximation using the SABON basis (middle), and using a Fourier basis of the same cardinality (right). In both cases, we use only the respective basis functions together with the true dynamics to solve the resulting eigenproblem, ensuring a fair comparison of basis quality. The densities are normalised to have unit $L^1(m)$ norm.}
    \label{fig:srb_plot}
\end{figure}

The ``ground truth'' SRB measure is constructed using a Galerkin method with $N=100$ complex Fourier modes per coordinate, yielding $N^2 = 10,000$ basis functions in total. The transfer operator matrix entries are computed via FFT on a fine $4N \times 4N = 400 \times 400$ spatial grid to ensure accurate quadrature. The leading eigenfunction is then evaluated on the $100 \times 100$ grid used throughout our experiments.
In comparison, our approximated SRB measure uses 324 learned basis functions. We also compare against a Fourier basis of the same cardinality. We use the real Fourier basis on $[0,1]^2$ constructed as a tensor product of 1D bases with 18 functions each: one constant, sine-cosine pairs for modes of order $1, \ldots, 8$, and a single cosine at order $9$ (the corresponding sine is omitted to obtain exactly 324 functions for a fair comparison). The tensor product yields $18^2 = 324$ basis functions in total.
From Figure \ref{fig:srb_plot} we can see that our reconstructed SRB measure using our learned basis clearly captures more detail and relevant dynamical structure than a high-quality basis created from a similar number of Fourier modes.

An SRB measure is in general a distribution and not a function, and therefore it has no density with respect to Lebesgue measure. This makes comparisons in the $L^2$ norm less meaningful. 
Negative index Sobolev spaces \cite{MMP05,Thiffeault12} have been used to describe chaotic mixing because they contain distributions of certain orders and their norms are simple to evaluate.
We choose the space $H^{-1}$ as advocated in \cite{Thiffeault12} because it is in spirit compatible with the lowest-order Banach space from \cite{GL06}.
For a function $f$ on the torus $\mathbb{T}^d$ with Fourier coefficients $\hat{f}_k$, the $H^{-1}$ norm is defined as
\begin{align*}
    \|f\|_{H^{-1}}^2
    \;\coloneq\;
    \sum_{k \in \mathbb{Z}^d} (1 + |k|^2)^{-1} |\hat{f}_k|^2,
\end{align*}
where $|k| = \sqrt{k_1^2 + \cdots + k_d^2}$ is the Euclidean norm of the multi-index $k$. 
Using the above expression, the $H^{-1}$ norm is straightforward to evaluate via the FFT.
This norm penalises high-frequency components less severely than $L^2$, making it natural for comparing distributions that may differ in fine-scale structure.

To determine which numerical eigenfunctions of $\hat{\mathcal{L}}$ are meaningful to study, we carry out the heuristic Test 1 \cite{FGTQ14} by  computing the ratios $r_k = \frac{||\psi_k||_{H^{-1}}}{||\psi_k||_{L^{2}}}$ of
     the eigenfunctions $\{\psi_k\}$ of the approximate transfer operator $\hat{\mathcal{L}}$.
If this ratio is large, this indicates that $\psi_k\in H^{-1}(X)$, while if this ratio is very small, the opposite is indicated.
We find a very clear dichotomy in these ratios:  only the dominant eigenfunction $\psi_1$,  corresponding to the eigenvalue 1 and approximating $\mu_{\rm SRB}$, exhibits a large ratio, while all other ratios are approximately zero, implying that only the leading eigenfunction lies in $H^{-1}$. Residual computations (see e.g.\ \cite{CT23}) made in $L^2(X)$ also show a lower residual for the eigenvalue 1.
We therefore limit our quantitative analysis to the SRB measure.

To quantitatively benchmark the learned basis against the traditional Fourier basis, we measure two types of error. First, the \emph{projection error}, which will quantify how well each basis can represent the ground-truth SRB measure under both $L^2$ and $H^{-1}$ orthogonal projection in their respective norms. Second, the \emph{approximation error} compares the estimated SRB measure $\mu_\mathcal{B}$ constructed as an eigenvector of $\mathcal{L}_\mathcal{B}$ (Galerkin projection of $\mathcal{L}$ onto  the learned basis $\mathcal{B}$) against the ground truth in both the $L^2$ and $H^{-1}$ norms.

Letting $\mathcal B$ denote either of the two bases, denote the $L^2$-orthogonal projector onto $\mathcal B$ by $P_{\mathcal B}:L^{2}(X)\to\operatorname{span}(\mathcal B)$. Define the relative $L^2$ projection error
\begin{equation}
\label{eq:projection_error_l2}
    L^{2, \mathrm{rel}}_{\mathcal B}\bigl(\mu_{\mathrm{SRB}},P_{\mathcal B}(\mu_{\mathrm{SRB}})\bigr)
    \;\coloneq\;\frac{\| \mu_{\mathrm{SRB}}
    - P_{\mathcal B}(\mu_{\mathrm{SRB}}) \|_{2}}
    {\| \mu_{\mathrm{SRB}} \|_{2}},
\end{equation}
and define the relative $L^2$ approximation error for $\mu_B$
\begin{equation}
\label{eq:approximation_error_l2}
    L^{2, \mathrm{rel}}\bigl(\mu_{\mathrm{SRB}},{\mu}_{\mathcal B}\bigr)
    \;\coloneq\;\frac{\| \mu_{\mathrm{SRB}}
    - {\mu}_{\mathcal B} \|_{2}}
    {\| \mu_{\mathrm{SRB}} \|_{2}},
\end{equation}

We define the relative $H^{-1}$ projection error in an identical way to the $L^2$ projection error, except the projection is $H^{-1}$-orthogonal and the norm is $H^{-1}$:
\begin{equation}
\label{eq:projection_error_h1}
    H^{-1, \mathrm{rel}}_{\mathcal B}\bigl(\mu_{\mathrm{SRB}},P_{\mathcal B}(\mu_{\mathrm{SRB}})\bigr)
    \;\coloneq\;\frac{\| \mu_{\mathrm{SRB}}
    - P_{\mathcal B}(\mu_{\mathrm{SRB}}) \|_{H^{-1}}}
    {\| \mu_{\mathrm{SRB}} \|_{H^{-1}}}.
\end{equation}
The relative $H^{-1}$ approximation error for $\mu_{\mathcal{B}}$ is
\begin{equation}
    \label{eq:approximation_error_h1}
    H^{-1, \mathrm{rel}}\bigl(\mu_{\mathrm{SRB}},{\mu}_{\mathcal B}\bigr)
    \;\coloneq\;\frac{\| \mu_{\mathrm{SRB}}
    - {\mu}_{\mathcal B} \|_{H^{-1}}}
    {\| \mu_{\mathrm{SRB}} \|_{H^{-1}}}.
\end{equation}
We compute these quantities when (i) $\mathcal{B}$ is the learned basis and (ii) $\mathcal{B}$ is a Fourier basis of comparable cardinality.
\cref{tab:projection_error_L2_H-1} shows that in both relative $L^{2}$ and $H^{-1}$ norms, the learned SABON basis has a consistently smaller error (by a factor of around four) for both projection and approximation, further illustrating that the learned basis is able to better capture the detailed structure in the ground truth SRB distribution.

\begin{table}[H]
  \caption{Approximation quality of the ground-truth SRB measure for the weakly nonlinear cat map. Projection error (\cref{eq:projection_error_l2,eq:projection_error_h1}) measures how well each basis can represent the SRB measure under orthogonal projection. Approximation error (\cref{eq:approximation_error_l2,eq:approximation_error_h1}) compares the estimated SRB measure $\mu_{\mathcal{B}}$, computed by solving ${\mathcal{L}_\mathcal{B}}\mu_{\mathcal{B}}=\mu_{\mathcal{B}}$, with the ground-truth SRB measure.}
\label{tab:projection_error_L2_H-1}
\centering
\begin{tabular}{|l|c|c|c|c|}
    \hline
    & \multicolumn{2}{c|}{Projection Error} & \multicolumn{2}{c|}{Approximation Error} \\
    \cline{2-5}
    Basis ($N=324$) & $L^{2, \mathrm{rel}}$ & $H^{-1, \mathrm{rel}}$ & $L^{2, \mathrm{rel}}$ & $H^{-1, \mathrm{rel}}$ \\ \hline
    SABON & $6.436\times10^{-2}$ & $1.016\times10^{-3}$ & $7.314\times10^{-2}$ & $4.064\times10^{-3}$ \\ \hline
    Fourier & $2.647\times10^{-1}$ & $1.576\times10^{-2}$ & $2.648\times10^{-1}$ & $1.578\times10^{-2}$ \\ \hline
\end{tabular}
\end{table}

\subsubsection{A strongly conjugated nonlinear perturbation of Arnold's cat map}
\label{subsec:conjugated_cat_map}
To test SABON on dynamics with more complex geometry, we consider a smooth conjugation of the perturbed cat map from \cref{subsec:nonlinearly_cat_map}. Let $F:\mathbb{T}^2 \to \mathbb{T}^2$ be the diffeomorphism
\begin{equation*}
    F(x,y) = \bigl(x - a\sin(2\pi x),\; y + b\sin(2\pi y + \pi/4)\bigr) \pmod{1},
\end{equation*}
with $a = b = 0.1$. We define the conjugated map
\begin{equation*}
    T_{\mathrm{conj}} \coloneq F^{-1} \circ T \circ F,
\end{equation*}
where $T$ is the perturbed cat map from \cref{subsec:nonlinearly_cat_map} with $\delta = 0.01$.
Since $T_{\mathrm{conj}}$ is diffeomorphically conjugate to $T$, it remains Anosov and possesses an SRB measure $\mu_{\rm SRB_\mathrm{conj}} = (F^{-1})_* \mu_{\rm SRB}=\mu_{\rm SRB}\circ F$, where $\mu_{\rm SRB}$ is the SRB measure of $T$.

The transfer operator for $T_{\mathrm{conj}}$ is given by
\begin{equation*}
    \mathcal{L}_{\mathrm{conj}} f = \frac{f \circ T_{\mathrm{conj}}^{-1}}{|\det DT_{\mathrm{conj}} \circ T_{\mathrm{conj}}^{-1}|},
\end{equation*}
where the Jacobian determinant of $T_{\mathrm{conj}}$ at a point $z$ satisfies (by the chain rule and the inverse function theorem)
\begin{equation*}
    |\det DT_{\mathrm{conj}}(z)| = \frac{|\det DT(F(z))| \cdot |\det DF(z)|}{|\det DF(T_{\mathrm{conj}}(z))|}.
\end{equation*}

We train SABON on $T_{\rm conj}$ using the same architecture and training procedure as in \cref{subsec:nonlinearly_cat_map}, but with a larger basis of $N=676$ functions to accommodate the more complex phase space geometry. For comparison, we use a Fourier basis constructed as in the previous example with 26 modes per dimension (maximum order 13), yielding $26^2 = 676$ functions. \cref{fig:conjugated_input_output} displays the model's performance on an unseen observable.
\begin{figure}[H]
\centering
\includegraphics[width=1\linewidth]{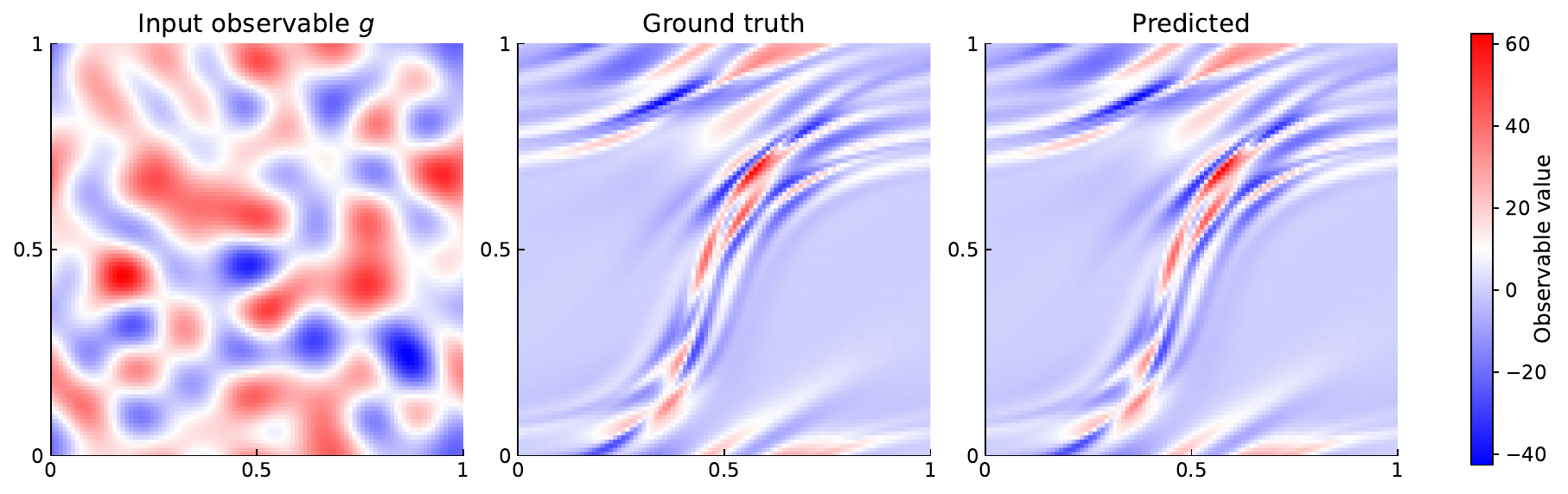}
\caption{Input and output for the Perron–Frobenius operator of the conjugated cat map. Left to right: (i) an unseen input observable $g$, (ii) ground-truth output, (iii) predicted output. The mean $L^{2}$ relative error over the test set is $1.736 \times 10^{-2}$.}
\label{fig:conjugated_input_output}
\end{figure}

The learned basis functions, shown in \cref{fig:conjugated_basis_functions}, exhibit clear adaptation to the geometry of the stable and unstable directions. As in the previous examples, the learned basis is approximately orthogonal; we omit the Gram matrix for brevity. Unlike the weakly nonlinear case in \cref{subsec:nonlinearly_cat_map}, where the basis functions display nearly linear striations approximately aligned with the linear cat map's eigendirections, the basis functions here show pronounced curvature reflecting the distortion introduced by the conjugacy $F$. This demonstrates SABON's ability to automatically discover basis functions adapted to the underlying dynamical geometry, even when that geometry is significantly more complex than the linear case.
\begin{figure}[H]
\centering
\includegraphics[width=1\linewidth]{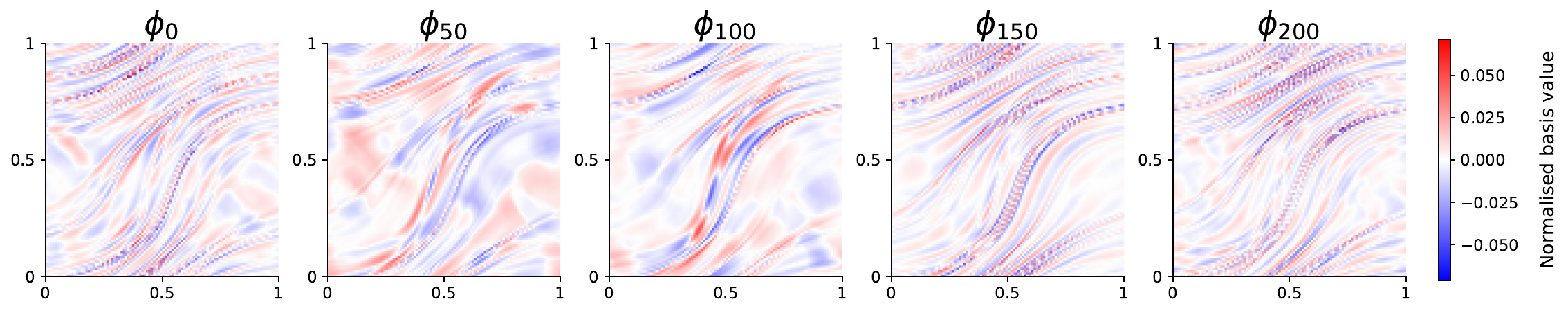}
\caption{Representative basis functions learned for the conjugated cat map. The curved structures reflect the nonlinear geometry of the stable and unstable directions induced by the conjugacy.}
\label{fig:conjugated_basis_functions}
\end{figure}

Solving the eigenequation \eqref{eq:eigenproblem} as in the previous examples, \cref{fig:conjugated_eigenvalues} shows the numerical eigenvalues of $\hat{\mathcal{L}}$. Despite the increased geometric complexity of the conjugated map, the spectral structure remains well-resolved: the leading eigenvalue lies close to unity, confirming the presence of an invariant measure, while the remaining eigenvalues are contained within a disk away from the unit circle.

To compare the quality of the SABON and Fourier bases independently of the learned operator, we compute the SRB measure as the leading eigenvector of the transfer operator matrix constructed using the true dynamics in both cases. That is, for each basis we project the true transfer operator onto the respective subspace and solve the resulting eigenequation. \cref{fig:conjugated_srb} shows that the SABON basis captures the fine-scale structure of the invariant measure more accurately than the Fourier basis, which provides a smoother approximation with less detail in the curved features induced by the conjugacy.

\begin{figure}[H] 
\centering 
\includegraphics[width=0.8\linewidth]{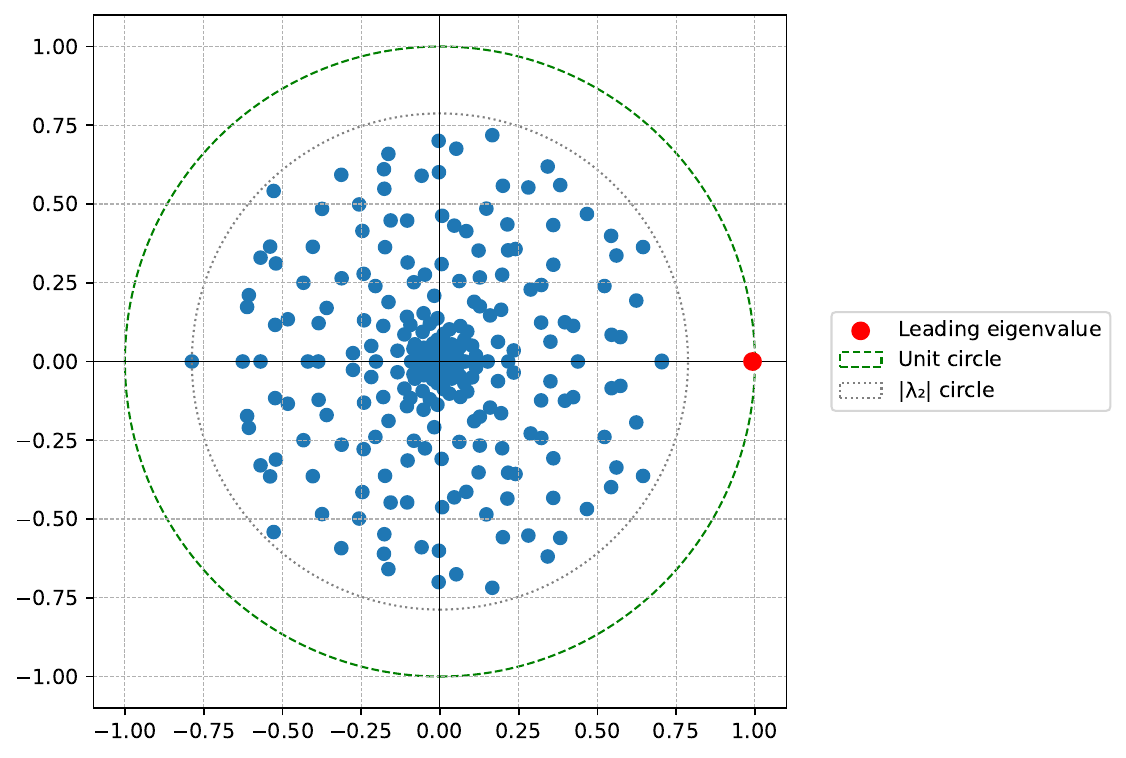} 
\caption{Eigenvalues of the learned Perron–Frobenius operator for the strongly nonlinear conjugated cat map. As in the weakly nonlinear case, the leading eigenvalue (corresponding to the SRB measure) is close to 1, and the remaining eigenvalues exhibit a spectral gap.} 
\label{fig:conjugated_eigenvalues} 
\end{figure}

\begin{figure}[H]
\centering
\includegraphics[width=1\linewidth]{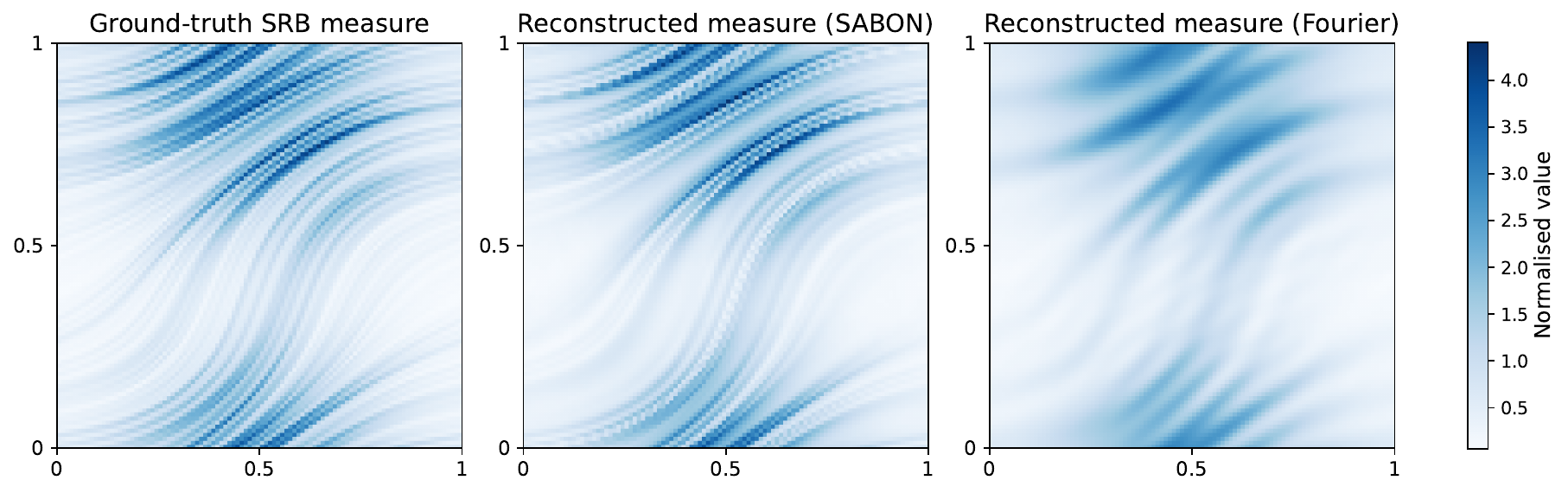}
\caption{Densities of ground-truth SRB measure for the conjugated map (left), its approximation using the SABON basis (middle), and using a Fourier basis of the same cardinality (right). In both cases, we use only the respective basis functions together with the true dynamics to solve the resulting eigenproblem, ensuring a fair comparison of basis quality. The densities are normalised to have unit $L^1(m)$ norm.}
\label{fig:conjugated_srb}
\end{figure}

Following the same procedure as in \cref{subsec:nonlinearly_cat_map}, we report the projection and approximation errors in \cref{tab:conjugated_projection_error}.

\begin{table}[H]
\caption{Approximation quality of the ground-truth SRB measure for the strongly nonlinear conjugated cat map. Projection error (\cref{eq:projection_error_l2,eq:projection_error_h1}) measures how well each basis can represent the SRB measure under orthogonal projection. Approximation error (\cref{eq:approximation_error_l2,eq:approximation_error_h1}) compares the estimated SRB measure $\mu_{\mathcal{B}}$, computed by solving ${\mathcal{L}_\mathcal{B}}\mu_{\mathcal{B}}=\mu_{\mathcal{B}}$, with the ground-truth SRB measure.}
\label{tab:conjugated_projection_error}
\centering
\begin{tabular}{|l|c|c|c|c|}
    \hline
    & \multicolumn{2}{c|}{Projection Error} & \multicolumn{2}{c|}{Approximation Error} \\
    \cline{2-5}
    Basis ($N=676$) & $L^{2, \mathrm{rel}}$ & $H^{-1, \mathrm{rel}}$ & $L^{2, \mathrm{rel}}$ & $H^{-1, \mathrm{rel}}$ \\ \hline
    SABON & $6.091\times10^{-2}$ & $1.466\times10^{-3}$ & $1.147\times10^{-1}$ & $6.225\times10^{-3}$ \\ \hline
    Fourier & $2.409\times10^{-1}$ & $1.095\times10^{-2}$ & $2.414\times10^{-1}$ & $1.132\times10^{-2}$ \\ \hline
\end{tabular}
\end{table}

Fourier modes have many strong approximation properties and are well suited to the periodic domains we have considered. Why then does the SABON basis consistently produce more accurate estimates of $\mu_{\rm SBR}$ than a Fourier basis?
Firstly, a Fourier element --- whether the complex exponential $e^{2\pi i k\cdot x}$ or the purely real sine–cosine products introduced in \cref{subsec:data_generation} --- is aligned with coordinate axes.
Secondly, the contracting dynamics of $T$ means that the frequency of a Fourier mode will be increased in stable directions of the dynamics under application of $\mathcal{L}$.
Thus, those modes in the Fourier basis with higher frequencies may be pushed far out of the original span of the Fourier modes, and made near-orthogonal to the Fourier basis.

The SABON basis, on the other hand, is learned jointly with the dynamics, so the network discovers functions that organise along the stable and unstable directions. The resulting modes resemble elongated ridges: smooth in the expanding direction and more sharply confined in the contracting one. 
When such a ridge is pushed forward by $\mathcal{L}$ it tends to become another ridge.
This anisotropic feature, absent from the isotropic Fourier family, accounts for the superior approximation quality reported in \cref{tab:projection_error_L2_H-1} and \cref{tab:conjugated_projection_error}.

A further crucial aspect of our learned SABON basis is its \textit{efficiency}.
As mentioned above, our learned basis functions tend to align with the stable and unstable directions arising from the dynamics.
This yields basis functions that have more rapid oscillation in stable directions and slower oscillation in unstable directions, which is exactly the form one expects for eigendistributions \cite{GL06}.
With a comparable number of basis functions, our learned functions can therefore more accurately represent the eigenprojectors. 

\section{Conclusions}
\label{sec:conclusions}
We introduced the Single Autoencoder Basis Operator Network
(SABON)—a neural framework that jointly learns
\begin{itemize}
\item an operator approximation $\hat{\mathcal{L}}$ of a Koopman or
Perron–Frobenius operator~$\mathcal L$ and 
\item a data-driven basis $\{\phi_j\}_{j=1}^{N}$ whose span
$V=\operatorname{span}\{\phi_j\}_{j=1}^{N}$ is adapted to the underlying dynamics and provides an efficient finite-dimensional representation for $\hat{\mathcal{L}}$ and other approximations of $\mathcal{L}$.
\end{itemize}
SABON cleanly achieves this by combining an encoding–reconstruction structure with a linear latent map~$\mathcal{G}$ and a composite loss that balances operator approximation accuracy and sparsity.
We proved a universal approximation theorem in the spirit of
Hua~\cite{HUA202321}, specialised to a single neural basis that targets transfer and Koopman operators.

Three canonical examples were used to highlight key features of SABON:

\begin{enumerate}
    \item \textbf{Circle rotation.} The SABON-learned basis accurately reconstructed the analytic eigenpairs and recovered the rotation angles~$k\alpha$, $k\in\mathbb{Z}$, directly from the leading eigenvalues, while the associated eigenfunctions coincided with the theoretical ${e^{\,i k\theta}}$ modes. Thus, the learned basis accurately captured the main dynamical properties. In this setting, where the training data lies in a finite-dimensional invariant subspace, SABON recovers an invariant basis.
    \item \textbf{Weakly nonlinear perturbed cat map.} For this uniformly hyperbolic map with approximately linear stable and unstable foliations, the learned basis adapted to the geometry of the dynamics and outperformed a Fourier basis of the same size when approximating the SRB measure.
    \item \textbf{Strongly nonlinear conjugated cat map.} For a more geometrically demanding Anosov map with curved foliations and significant deviation from area preservation, SABON continued to produce basis functions adapted to the nonlinear geometry and outperformed Fourier bases. This is particularly surprising given the well-known theoretical approximation strengths of Fourier modes. The network discovered basis functions largely aligned with the contracting and expanding directions, leading to more accurate estimates of the physical invariant measure of the dynamics.
\end{enumerate}

SABON opens up new possibilities for constructing compact and accurate transfer and Koopman operator representations, particularly in difficult situations where isotropic bases fail and data-adaptive bases may excel. Moreover, SABON is in principle extendable to high-dimensional or partially observed systems. Future theoretical challenges include deriving quantitative rates that relate network width, depth, and sparsity to the operator approximation error.

{\section*{Data availability}
All code and data supporting our findings is publicly available in our repository at \url{https://github.com/kevinkuhl/SABON}.}

\section*{Acknowledgments}

The authors thank Irteza Chaudhry for helpful discussions and for pointing out an error in an earlier version.
The research of GF is supported by an Australian Research Council (ARC) Laureate Fellowship FL230100088. 
The research of KK is supported by a UNSW University International Postgraduate Award, with additional funding from UNSW's School of Mathematics and Statistics and the ARC Laureate Fellowship. 
Computational resources were generously provided by Katana infrastructure at UNSW \cite{katana}, which significantly facilitated this research.

\bibliographystyle{siamplain}
\bibliography{references}
\end{document}